\documentclass[a4paper,10pt]{article}
\usepackage[utf8]{inputenc}
\usepackage{amsmath,amssymb,amsthm}
\usepackage{mathtools}
\usepackage{stmaryrd}
\usepackage{enumerate}
\usepackage[mathscr]{eucal}
\usepackage{anysize}
\usepackage{multirow}
\usepackage{cleveref}
\usepackage{wasysym}    % for defining command \rightarrowTriangle

\newcommand{\kron}{\ensuremath{\boxtimes}} % Kronecker
\newcommand{\T}{{\sf T}}        % transposition

\DeclareMathOperator{\rank}{rank}
\DeclareMathOperator{\trace}{trace}

\DeclareMathOperator{\GL}{GL}

\DeclareMathOperator{\sspan}{span}

\makeatletter
\def\rddots{\mathinner{\mkern1mu\raise\p@
\vbox{\kern7\p@\hbox{.}}\mkern2mu
\raise4\p@\hbox{.}\mkern2mu\raise7\p@\hbox{.}\mkern1mu}}
\makeatother

\newcommand{\rightarrowTriangle}{\relbar\joinrel\mathrel\RHD}

\DeclareFontFamily{U}{BOONDOX-calo}{\skewchar\font=45 }
\DeclareFontShape{U}{BOONDOX-calo}{m}{n}{
  <-> s*[1.05] BOONDOX-r-calo}{}
\DeclareFontShape{U}{BOONDOX-calo}{b}{n}{
  <-> s*[1.05] BOONDOX-b-calo}{}
\DeclareMathAlphabet{\mathscrc}{U}{BOONDOX-calo}{m}{n}

\theoremstyle{definition}
\newtheorem{defn}{Definition}
\newtheorem{lem}[defn]{Lemma}
\newtheorem{prop}[defn]{Proposition}
\newtheorem{theo}[defn]{Theorem}

\newtheorem{corol}[defn]{Corollary}

\theoremstyle{remark}
\newtheorem{rem}[defn]{Remark}
\newtheorem{exmp}[defn]{Example}

\usepackage{pgfplots}
\usetikzlibrary{external}
\usetikzlibrary{shapes,arrows,arrows.meta,chains,fit}

\tikzset{%
  highlight/.style={rectangle,rounded corners,fill=red!15,draw,
    fill opacity=0.5,thick,inner sep=0pt}
}

\makeatletter
\def\addlegendimage{\pgfplots@addlegendimage}
\makeatother

\newcommand\blfootnote[1]{%
  \begingroup
  \renewcommand\thefootnote{}\footnote{#1}%
  \addtocounter{footnote}{-1}%
  \endgroup
}

\title{On the minimal ranks of matrix pencils and the existence of a best approximate block-term tensor decomposition}
\author{J.~H.~de M.~Goulart$^{\dagger*}$ \, and \, P.~Comon$^\dagger$ \\[5pt]
        {\small $\dagger$ Univ.~Grenoble Alpes, CNRS, Gipsa-Lab, F-38000 Grenoble, France } \\ 
        {\small $*$ I3S Laboratory, CNRS, Univ.~C\^{o}te D'Azur, 06900 Sophia-Antipolis, France} \\
        {\small \{henrique.goulart, pierre.comon\}@gipsa-lab.fr}
}

\begin{document}

\maketitle

\blfootnote{$*$ Corresponding author.}
\blfootnote{This work was supported by the European Research Council under the European Programme FP7/2007-2013, Grant AdG-2013-320594 
``DECODA.''}
\vspace{-1cm}

\abstract{
Under the action of the general linear group with tensor structure, the ranks of matrices $A$ and $B$ forming an $m \times n$ 
pencil $A + \lambda B$ can 
change, but in a restricted manner. Specifically, with every pencil one can associate a pair of minimal ranks, which is unique 
up to a 
permutation. This notion can be defined for matrix pencils and, more generally, also for matrix polynomials of arbitrary degree. In 
this paper, we provide a formal definition of the minimal ranks, discuss its properties and the natural hierarchy it 
induces in a pencil space.
Then, we show how the minimal ranks of a pencil can be determined from its Kronecker canonical form.
For illustration, we classify the orbits according to their minimal ranks (under the action of the general linear group) in 
the case of real pencils with $m, n \le 4$. 
Subsequently, we show that real regular $2k \times 2k$ pencils having only complex-valued eigenvalues, which form an open 
positive-volume set, do not admit a best approximation (in the norm topology) on the set of real pencils whose minimal ranks are 
bounded by $2k-1$.
Our results can be interpreted from a tensor viewpoint, where the minimal ranks of a degree-$(d-1)$ matrix polynomial characterize the 
minimal ranks of matrices constituting a block-term decomposition of an $m \times n \times d$ tensor into a sum of matrix-vector tensor 
products.

\vskip4mm

}

\section{Introduction}
\label{sec:intro} 

% Kronecker-Weierstrass theory
% - Group of equivalence transformations
% - Elementary divisors characterizing pencils

The Kronecker-Weierstrass theory of $m \times n$ matrix pencils provides a complete classification in terms of 
$\GL_{m,n}(\mathbb{F})$-orbits, which are equivalence classes under the action of $\GL_{m}(\mathbb{F}) \times \GL_{n}(\mathbb{F})$:
\begin{equation*}
 \left( P, \, U \right) \cdot \left( \mu A + \lambda B \right) =   \mu P A U^\T + \lambda P B U^\T,
\end{equation*}
where the pencil $(A,B)$ is expressed in homogeneous coordinates. 
Here, $\mathbb{F}$ denotes a field, usually $\mathbb{R}$ or $\mathbb{C}$.
These orbits are represented by Kronecker canonical forms, which are characterized by unique minimal indices describing the singular 
part of the pencil along with elementary divisors associated with its regular part \cite{Gantmacher1960}. 
It follows then that these elementary divisors are $\GL_{m}(\mathbb{F}) \times \GL_{n}(\mathbb{F})$-invariant.

% Additional contributions by Ja'Ja' and Atkinson - tensor equivalence classes
% - Preservation of regularity and of minimal indices

This theory has been extended by Ja'Ja' and Atkinson \cite{JaJa1979b, Atkinson1991}, who have characterized the orbits of the larger 
group of tensor equivalence transformations $\GL_{m}(\mathbb{F}) \times \GL_{n}(\mathbb{F}) \times \GL_2 (\mathbb{F})$, acting on 
pencils via
\begin{equation*}
 \left( P, \, U, \, T \right) \cdot \left( \mu A + \lambda B \right) 
  =  \mu P (t_{11} A + t_{12} B) U^\T + \lambda P (t_{21} A + t_{22} B) U^\T,
\end{equation*} 
where 
\begin{equation*}
 T = \begin{pmatrix}
      t_{11} & t_{12} \\ t_{21} & t_{22}
     \end{pmatrix}
     \in \GL_2 (\mathbb{F}).
\end{equation*} 
For simplicity of notation, we use the shorthands $\GL_{m,n,2}(\mathbb{F}) \triangleq \GL_{m}(\mathbb{F}) \times 
\GL_{n}(\mathbb{F}) \times \GL_2 (\mathbb{F})$ and $\GL_{m,n}(\mathbb{F}) \triangleq \GL_{m}(\mathbb{F}) \times \GL_{n}(\mathbb{F})$. 
Ja'Ja' \cite{JaJa1979b} has shown that the Kronecker minimal indices of a pencil are $\GL_{m,n,2}(\mathbb{F})$-invariant, and so the 
singular part of a pencil is preserved by $\GL_{m,n,2}(\mathbb{F})$ as well. 
However, the elementary divisors of its regular part are not. Nevertheless, their powers still remain the same, which motivates the 
terminology ``invariant powers,'' used by Ja'Ja' \cite{JaJa1979b}. Atkinson \cite{Atkinson1991} went on to prove that, for an 
algebraically closed field $\mathbb{F}$, the equivalence classes of regular pencils are characterized by those powers and also by 
certain ratios which completely describe the elementary divisors. Specifically, recalling that over such a field all elementary 
divisors are powers of linear factors of the form
\begin{equation*}
  \phi_i(\mu, \lambda) = \alpha_i \mu + \beta_i \lambda  \quad \text{for some } \alpha_i, \beta_i \in \mathbb{F},
\end{equation*}
these ratios are defined as $\gamma_i \triangleq \alpha_i / \beta_i \in \mathbb{F} \cup \{\infty\}$.

% - Connection with tensor rank (which is invariant under the action of GL_{m,n,2})

When viewed as a tensor, the (tensor) rank\footnote{This is not to be confused with the \emph{normal rank} of the pencil 
$\mu A + \lambda B$, simply defined as $\rank(\mu A + \lambda B)$ \cite{Gantmacher1960, Ikramov1993}.} of a pencil $\mu A + \lambda B$ 
is defined as the minimal number $r$ of rank-one matrices $U_1, \ldots, U_r \in \mathbb{F}^{m \times n}$ such that $A, B \in \sspan 
\{U_1, \ldots, U_r\}$ \cite{JaJa1979}. Equivalently, it is given by the 
minimal number $r$ such that one can find vectors $u_i \in \mathbb{F}^m$, $v_i \in \mathbb{F}^n$ and $w_i \in \mathbb{F}^2$ satisfying
\begin{equation*}
 A \otimes e_1 + B \otimes e_2 = \sum_{i=1}^r u_i \otimes v_i \otimes w_i,
\end{equation*} 
where $e_i$ denotes the canonical basis vector of $\mathbb{F}^2$. Under the action of $(P,U,T) \in \GL_{m,n,2}(\mathbb{F})$, this 
expression is transformed into the tensor
\begin{equation*}
 \left( P, U, T \right) \cdot \left( A \otimes e_1 + B \otimes e_2 \right) 
   = \sum_{i=1}^r  \left(  P \, u_i \right) \otimes \left( U \, v_i \right) \otimes \left( T \, w_i \right),
\end{equation*} 
from which it is visible that the tensor rank is $\GL_{m,n,2}(\mathbb{F})$-invariant. The rank of $m \times n \times 2$ tensors 
can thus be studied by considering $\GL_{m,n,2}(\mathbb{F})$-orbits and associated representatives (see, e.g., the classification of 
$\GL_{2,2,2}(\mathbb{R})$-orbits undertaken by De Silva and Lim \cite{deSilva2008}).

% - Application: algebraic complexity (simultaneous evaluation of pairs of bilinear forms)
% - Results by Ja'Ja' and Sumi et al (completely characterize the rank of m x n x 2 tensors over R and C).

One application of the study of $\GL_{m,n,2}(\mathbb{F})$-orbits is in algebraic complexity theory, since the tensor rank 
of $\mu A + \lambda B$ quantifies the minimal number of multiplications needed to simultaneously evaluate a pair of 
bilinear forms $g_1(x,y) = \langle x, A y \rangle$ and $g_2(x,y) = \langle x, B y \rangle$ \cite{JaJa1979}. In the case where $\mathbb{F}$ is 
algebraically closed, Ja'Ja' \cite{JaJa1979} has derived results which allow determining the tensor rank of any pencil based on its 
Kronecker canonical form. Sumi et al.~\cite{Sumi2009} have extended these results to pencils over any field $\mathbb{F}$.

% Multilinear transformations are employed in Kronecker-Weierstrass theory to handle ``infinite elementary divisors''
% - Employed for regular pencils formed by a pair of singular matrices
% - Only possible because: every regular pencils is GL_{m,n,2}-equivalent to another pencil having one nonsingular matrix (trivial)
Tensor equivalence transformations  can also be employed to avoid so-called infinite elementary divisors of regular 
pencils \cite{Ikramov1993}. These arise when matrix $B$ is singular (note that both $A$ and $B$ can be singular but still 
satisfy $\det(\mu A + \lambda B) \not\equiv 0$). In non-homogeneous coordinates, the polynomial $\det (A + \lambda B)$ of an $n 
\times n$ pencil $A + \lambda B$ has degree $s = \rank B$, and its characteristic polynomial is said to have infinite 
elementary divisors of combined degree $n - s$. In this case,
the tensor equivalence transformation $A + \lambda B \mapsto B + \lambda (A + \alpha B)$ can be employed for any $\alpha \in 
\mathbb{F}\setminus\{0\}$ such that $\rank (A + \alpha B) = n$, yielding a pencil having only finite elementary divisors, 
including some of the form $\lambda^q$ induced by the infinite elementary divisors of $A + \lambda B$.
The existence of such an $\alpha$ is guaranteed by definition, since $A + \lambda B$ is regular. In other words, every regular pencil 
$\mu A + \lambda B$ is $\GL_{m,n,2}(\mathbb{F})$-equivalent to another pencil $\mu A' + \lambda B'$ such that $B'$ is 
nonsingular. In fact, it is always $\GL_{m,n,2}(\mathbb{F})$-equivalent to some $\mu A' + \lambda B'$ with nonsingular 
matrices $A'$ and $B'$.

% Minimal ranks of matrix pencils
% - On the inverse direction, certain regular pencils are GL_{m,n,2}-equivalent to other pencils formed by singular matrices
% - However, this is not true for all regular pencils 
% - Complementarily to the rank of a pencil, one can thus study another way of characterizing its intrinsic complexity: the minimal 
%   ranks of a pencil
On the other hand, not every regular pencil $\mu A + \lambda B$ constituted by nonsingular matrices $A$ and $B$ is 
$\GL_{n,n,2}(\mathbb{F})$-equivalent to some other pencil $\mu A' + \lambda B'$ such that either $A'$ or $B'$ are 
singular (or both). Take, for instance, $\mathbb{F} = \mathbb{R}$ and
\begin{equation*}
 A = \begin{pmatrix} 1 & 0 \\ 0 & 1 \end{pmatrix}, \quad B = \begin{pmatrix} 0 & -1 \\ 1 & 0 \end{pmatrix}.
\end{equation*} 
No tensor equivalence transformation in $\GL_{m,n,2}(\mathbb{R})$ of this pencil can yield $\mu A' + \lambda B'$ such that 
either $A'$ or $B'$ is singular. Obviously, this property is $\GL_{m,n,2}(\mathbb{R})$-invariant.
It turns out that each $\GL_{m,n,2}(\mathbb{F})$-orbit $\mathcal{O}$ of a matrix pencil space can be classified on the basis 
of its associated \emph{minimal ranks} $r$ and $s$, with $r \ge s$, which are such that every pencil $\mu A + \lambda B \in 
\mathcal{O}$ with $\rank A \ge \rank B$ satisfies $\rank A \ge r$ and $\rank B \ge s$.
This notion of intrinsic complexity of a matrix pencil is complementary to its tensor rank, in the sense that pencils of same 
tensor rank do not necessarily have the same minimal ranks and vice-versa. For simplicity, we will compactly denote the minimal ranks 
of a pencil by $ \rho(A,B) = (r,s)$.

% Connection with BTD
% - Considering again the connection with tensors, the minimal ranks define the intrinsic complexity of a pencil when decomposed in 
%   block terms, each term being a tensor product between a matrix and a vector
% - Theoretical properties of the BTD are therefore related with properties of matrix pencils via this notion of minimal ranks
% - We exploit this link to show that certain n x n x 2 tensors have no best BTD approximation with some chosen minimal ranks, 
% generalizing non-existence results related to the canonical polyadic decomposition (associated with the tensor rank)
It turns out that there is a direct connection between the minimal ranks of a pencil and the decomposition of its associated 
third-order tensor in block terms  consisting of matrix-vector outer products, as introduced by De Lathauwer 
\cite{deLathauwer2008b}. Namely, the components of $\rho(A,B)$ are the minimal numbers $r$ and $s$ satisfying
\begin{equation*}
 A \otimes e_1 + B \otimes e_2 = \left( \sum_{i=1}^{r} u_i \otimes v_i \right) \otimes w +
                                 \left( \sum_{i=1}^{s} x_i \otimes y_i \right) \otimes z,
\end{equation*} 
where $r \ge s$ and $\{w, z\}$ forms a basis for $\mathbb{F}^2$. The theoretical properties of such block-term decompositions 
(henceforth abbreviated as BTD) of $m \times n \times 2$ tensors are therefore related to properties of matrix pencils via this notion 
of minimal ranks. 

In this paper, we will more generally define the minimal ranks of $m \times n$ matrix polynomials $\sum_{k=1}^{d} \lambda^{k-1} 
A_k$, which include matrix pencils as a special case. This property of matrix polynomials is directly related to the BTD of $m 
\times n \times d$ tensors. In particular, similarly to the tensor rank, it induces a hierarchy of matrix polynomials, albeit a more 
involved one. 
We derive results which determine the minimal ranks of any matrix pencil in Kronecker canonical form.
A classification of $\GL_{m,n,2}(\mathbb{R})$-orbits of real $m \times n$ pencils in terms of their minimal ranks is then
carried out for $m, n \le 4$.
On the basis of these results, we proceed to show that:
\begin{enumerate}
 \item The set of real $2k \times 2k$ pencils which are $\GL_{2k,2k,2}(\mathbb{R})$-equivalent to some $\mu A + \lambda B$
with $\rank A \le 2k-1$ and $\rank B \le 2k-1$ is not closed in the norm topology for any positive integer $k$.
 \item No real $2k \times 2k$ pencil having minimal ranks $(2k,2k)$ admits a best approximation in the set above described. 
\end{enumerate}
The first above result is analogous to the fact that a set of tensors having rank bounded by some number $r>1$ is generally not 
closed. 
Similarly, the second one parallels the fact that no element of certain sets of real rank-$r$ tensors admits a best approximation of 
a certain rank\footnote{For instance, no real $2 \times 2 \times 2$ tensor of rank $3$ admits a best rank-$2$ approximation.} $r' < r$ 
in the norm topology \cite{Stegeman2006, deSilva2008}. 
This second result is of consequence to applications relying on the BTD, since the set of real pencils having minimal ranks $(2k,2k)$ 
is open in the norm topology, thus having positive volume. For complex-valued pencils, the results of \cite{Qi2017} imply that such a 
non-existence phenomenon can only happen over sets of zero volume. We shall give a template of examples of (possibly complex) pencils 
having no best approximation on a given set of pencils with strictly lower minimal ranks.

It should be noted that the fact that a tensor might not admit an approximate BTD with a certain prescribed structure (referring 
to the number of blocks and their multilinear ranks \cite{deLathauwer2008b}) is already known. Specifically, De Lathauwer 
\cite{deLathauwer2008c} has provided an example relying on a construction similar to that of De Silva and Lim \cite{deSilva2008} 
concerning the case of low-rank tensor approximation. Our example given in \Cref{sec:ill-posed-examples} is in the same spirit. 
Nonetheless, to our knowledge ours is the first work showing the existence of a positive-volume set of tensors having no approximate 
BTD of a given structure, a phenomenon which is known to happen for low-rank tensor approximation \cite{Stegeman2006, deSilva2008}.

% &&&&&&&&&&&&&&&&&&&&&&&&&&&&&&&&&&&&&&&&&&&&&&&&&&&&&&&&&&&&&&&&&&&&&&&&&&&&&&&&&&&&&&&&&&&&&&&&&&&&&&&&&&&&&&&&&&&&& %
% &&&&&&&&&&&&&&&&&&&&&&&&&&&&&&&&&&&&&&&&&&&&&&&&&&&&&&&&&&&&&&&&&&&&&&&&&&&&&&&&&&&&&&&&&&&&&&&&&&&&&&&&&&&&&&&&&&&&& %
 
\section{Minimal ranks of pencils}
\label{sec:pencils}

For brevity, we will henceforth express matrix pencils only in non-homogeneous coordinates.

\subsection{Definition and basic results}
\label{sec:min-ranks}

Given its prominent role in what follows, the $\GL_{m,n,2}(\mathbb{F})$-orbit of a pencil deserves a special notation:
\begin{equation*}
 \mathcal{O}(A,B) \triangleq 
                   \{(P,U,T) \cdot (A + \lambda B) \ | \ (P,U,T) \in \GL_{m,n,2}(\mathbb{F}) \}.
\end{equation*}
It will also be helpful to introduce the sets
\begin{equation}
\label{Bset}
 \mathcal{B}_{r,s} \triangleq 
                   \left\{ A + \lambda B \; | \;
                   \exists \, A' + \lambda B' \in \mathcal{O}(A,B) 
                   \mbox{ such that }
                   \rank A' \le r, \, \rank B' \le s \right\}.
\end{equation} 
Clearly, $\mathcal{B}_{r,s} = \mathcal{B}_{s,r}$, and thus we shall assume that $r \ge s$ without loss of generality.
For simplicity, when $r \ge s = 0$ we can also write $\mathcal{B}_r$ instead of $\mathcal{B}_{r,0}$. According to this convention and  
to definition \eqref{Bset}, we have, for instance, $\mathcal{B}_{r} \subseteq \mathcal{B}_{r,s} \subseteq \mathcal{B}_{r+1,s}$. 
Furthermore, $\mathcal{B}_{r,s}$ is by definition $\GL_{m,n,2}(\mathbb{F})$-invariant, because the relation $(A,B) \sim 
(A',B')$ defined as $A' + \lambda B' \in \mathcal{O}(A,B)$ is reflexive and transitive, i.e., it defines an equivalence class. Hence:

\begin{lem}
 $A + \lambda B \in \mathcal{B}_{r,s}$ if and only if $\mathcal{O}(A,B) \subseteq \mathcal{B}_{r,s}$.
\end{lem}	

As far as the question of whether $A + \lambda B$ is in $\mathcal{B}_{r,s}$ for some $(r,s)$ is concerned, all that matters is the 
action of $\GL_2(\mathbb{F})$. Indeed, if $A' + \lambda B' = (P,U,T) \cdot (A + \lambda B)$ are such that $\rank A' = r$ and $\rank B' 
= s$, then\footnote{Note that we denote the identity of $\GL_m(\mathbb{F})$ by $E_m$ or, when no ambiguity arises, simply by $E$.} $A'' 
+ \lambda B'' = (P^{-1},U^{-1},E) \cdot (A' + \lambda B') = (E,E,T) \cdot (A + \lambda B)$ satisfies $\rank A'' = \rank P^{-1} A' 
 U^{-\T} = \rank A'$ and $\rank B'' = \rank P^{-1} B' U^{-\T} = \rank B'$. We have shown the following.

\begin{lem}
\label{lem:GL2}
$A + \lambda B \in \mathcal{B}_{r,s}$ if and only if there exists $T \in \GL_2(\mathbb{F})$ such that $A' + \lambda B' = (E, E, T) 
\cdot (A + \lambda B)$ satisfies $\rank A' \le r$ and $\rank B' \le s$. In other words, $\mathcal{B}_{r,s}$ can be equivalently 
defined as the set of all pencils which are $\GL_2(\mathbb{F})$-equivalent to some $A' + \lambda B'$ satisfying $\rank A' \le r$ and 
$\rank B' \le s$.
\end{lem}

Let us now formally define the minimal ranks in terms of the introduced notation.

\begin{defn}
\label{def:min-ranks}
Let $A + \lambda B$ be an $m \times n$ pencil over $\mathbb{F}$. The \emph{minimal ranks} of $A + \lambda B$, denoted as 
$\rho(A,B) = (r,s)$, are defined as
\begin{align}
 \label{s-def}
s \triangleq & \ \min_{(t,u) \neq 0} \rank (t A + u B),  \\
r \triangleq & \  \min_{(t',u') \notin \sspan\{(t^\star,u^\star)\}} \rank(t' A + u' B), 
\label{r-def}
\end{align} 
where $(t^\star,u^\star)$ is any minimizer\footnote{Observe that such a minimizer is not unique. Besides the obvious family of 
minimizers of the form $(c t^\star, c u^\star)$, there may be also multiple non-collinear minimizers. For instance, for the pencil $E_4 
+ \lambda (a E_2 \oplus b E_2)$, with $a \neq b$, there are two non-collinear minimizers: $(-a,1)$ and $(-b,1)$.} of \eqref{s-def}. We 
obviously have $r\ge s$.
When denoting a pencil as $P(\lambda) = A + \lambda B$, we shall also use the notation $\rho(P) = \rho(A,B)$.
\end{defn}

The first thing to note is that $r$ is well-defined, i.e., its value is always the same regardless of the minimizer $(t^\star,u^\star)$
picked in the definition \eqref{r-def}. For different collinear minimizers of \eqref{s-def}, this is immediately clear. Now if two 
non-collinear minimizers $(t^\star,u^\star)$ and $(t^{\star\star}, u^{\star\star})$ exist for \eqref{s-def}, then $r = s$ must hold.
It is also clear from \eqref{s-def} and \eqref{r-def} that the minimal ranks of a pencil  $A + \lambda B$ are the ranks of 
matrices $A'$ and $B'$ of some pencil $A' + \lambda B'$ in the $\GL_{2}(\mathbb{F})$-orbit  of $A + \lambda B$. Indeed, 
taking $B' = t_{21} A + t_{22} B$ and $A' = t_{11} A + t_{12} B$, where $(t_{21},t_{22})$ and $(t_{11},t_{12})$ are the minimizers of 
\eqref{s-def} and \eqref{r-def}, respectively, then  $A + \lambda B = (E,E,T^{-1}) \cdot (A' + \lambda B')$, with $T = (t_{ij})  \in 
\GL_2(\mathbb{F})$.  Moreover, the minima in \eqref{s-def} and \eqref{r-def} are unchanged under a transformation from 
$\GL_{m,n}(\mathbb{F})$, implying that the value of $\rho$ is an invariant of $\GL_{m,n,2}$ action. Summarizing, we have:

\begin{prop}
 \label{prop:rho-orbit}
 Every $A' + \lambda B' \in \mathcal{O}(A,B)$ satisfies $\rho(A,B) = \rho(A',B')$. In particular, if $A' + \lambda B' \in 
\mathcal{O}(A,B)$ satisfies $(\rank A', \rank B') = \rho(A,B)$, we say that $A' + \lambda B'$ attains the minimal ranks of $A + \lambda 
B$.
\end{prop}

From the above discussion, $\rho(A,B) = (r,s)$ implies $A + \lambda B \in \mathcal{B}_{r,s}$. However, the converse is not true. For 
instance, $E_n + \lambda E_n \in \mathcal{B}_{n,n}$ but $\rho(E_n, E_n) = (n,0)$. In general, if $A = c B$ or $B = c A$ for 
some $c \in \mathbb{F}$ (including the possibility $c = 0$), then $\rho(A,B) = (r,0)$  with $r = \max\{\rank A, \rank B\}$. 
Conversely, $s = 0$ only if $A$ and $B$ are proportional. We thus have the following result. 

\begin{lem}
 \label{lem:rho-r-0}
A pencil $A + \lambda B$ satisfies $\rho(A,B) = (r,0)$ if and only if $A$ and $B$ are proportional. Furthermore, $r = 
\max\{\rank A, \rank B\}$.
\end{lem}

The terminology ``minimal ranks'' is motivated by the fact that, if $\rho(A,B) = (r,s)$ and $A + \lambda B \in 
\mathcal{B}_{r',s'}$ for 
some $r' \ge s'$, then \emph{both} $r' \ge r$ and $s' \ge s$ must hold. The definitions in \eqref{s-def} and \eqref{r-def} clearly 
imply $A + \lambda B \notin \mathcal{B}_{r',s'}$ for any pair $r'$ and $s' < s$ or any pair $r' < r$ and $s' = s$.
It remains to show that $A + \lambda B \notin \mathcal{B}_{r',s'}$ also for $r > r' \ge s' > s$. 
Suppose on the contrary that $A + \lambda B \in \mathcal{B}_{r',s'}$ with $r > r' \ge s' > s$. 
This implies there exists $T \in \GL_2\left( \mathbb{F} \right)$ such that 
\begin{equation*}
   \rank (t_{11} A + t_{12} B) \le r' \quad \text{and} \quad \rank (t_{21} A + t_{22} B) \le s', \quad 
   \text{with} \quad r > r' \ge s' > s.
\end{equation*}
But then, $r' > s$ implies $(t_{11},t_{12}) \notin \sspan\{(t^\star,u^\star)\}$, where $(t^\star,u^\star)$ is a minimizer of 
\eqref{s-def}. This contradicts the definition of $r$ given by \eqref{r-def}.

\begin{prop}
\label{prop:r-s}
If $\rho(A,B) = (r,s)$ and $r' \ge s'$, then $A + \lambda B \in \mathcal{B}_{r',s'}$ if and only if $r' \ge r$ and $s' \ge s$.
\end{prop}

We consider now some examples.

\begin{exmp}
A regular $n \times n$ pencil $A + \lambda B$ can only belong to $\mathcal{B}_r$ if $r = n$. Indeed, $A + \lambda B \in \mathcal{B}_r$ 
implies $A$ and $B$ are proportional,  say $B = \alpha A$, and $\rank(t A + u B) = \rank((t + \alpha u) A) \le r$ for any 
$(t,u) \in \mathbb{F}^2$. As a concrete example, $E + \lambda E$ is clearly in $\mathcal{B}_n$ but not in any $\mathcal{B}_{r'}$ with 
$r' < n$.
\end{exmp}

\begin{exmp}
Regular $n \times n$ pencils can also be in $\mathcal{B}_{r,s}$ for some $n > r \ge s > 0$. For example, the regular $3 \times 3$ 
pencil\footnote{The dimensions of the zero blocks in that expression should be clear from the context.} $E_2 \oplus 0 + \lambda (0 
\oplus E_1)$ is in $\mathcal{B}_{2,1}$. Yet, the constraint $r + s \ge n$ must be satisfied. Indeed, $A + \lambda B \in 
\mathcal{B}_{r,s}$ means $A + \lambda B$ is $\GL_{m\times n\times 2}(\mathbb{F})$-equivalent to some $A' + \lambda B'$, with $\rank A' 
\le r$ and $\rank B' \le s$. If $r + s < n$, then clearly $\det(A' + \lambda B') \equiv 0$, implying neither $A' + \lambda B'$ nor $A 
+ \lambda B$ is regular.
\end{exmp}

The next two examples underline how the elementary divisors of a regular pencil determine its minimal ranks. A general result 
establishing this connection will be presented ahead.

\begin{exmp}
 The pencil 
 \begin{equation*}
   Q + \lambda E =
             \begin{pmatrix} a & b  \\ -b & a \end{pmatrix}
   + \lambda \begin{pmatrix} 1 & 0  \\ 0 & 1 \end{pmatrix} 
 \end{equation*} 
 with $b \neq 0$ has minimal ranks $\rho(Q,E) = (2,2)$ in $\mathbb{R}$, because
 \begin{equation*}
             \rank \begin{pmatrix} u + t a & t b  \\ -t b & u + t a \end{pmatrix} = 2
 \end{equation*}  
for any  $(t,u) \in \mathbb{R}^2 \setminus \{0\}$. However, this is not true in $\mathbb{C}$, because the $\GL_2(\mathbb{C})$ 
transformation $(Q,E) \mapsto ((-a + bi) E + Q, E )$ yields $B + \lambda E$, where
\begin{equation*}
 B = \begin{pmatrix} bi & b \\ -b & bi \end{pmatrix} = 
      \begin{pmatrix} bi & bi \\ -b & -b \end{pmatrix} \begin{pmatrix} 1 & 0 \\ 0 & -i \end{pmatrix} 
\end{equation*} 
has rank 1. This difference comes from the fact that $Q + \lambda E$ has a single elementary divisor over $\mathbb{R}$, namely 
$\lambda^2 + 2a\lambda + (a^2 + b^2)$, which cannot be factored into powers of linear forms since its roots are complex. 
In fact, $Q$ is diagonalizable over $\mathbb{C}$, since it is similar to
\begin{equation*}
 B' = \begin{pmatrix} a + bi & 0 \\ 0 & a - bi \end{pmatrix}
\end{equation*}
From \Cref{prop:rho-orbit}, we have $\rho(Q,E) = \rho(B',E)$, and it is not hard to see that $\rho(B',E) = (1,1)$. 
\end{exmp}

\begin{exmp}
Defining
\begin{equation*}
 H  = \begin{pmatrix} a & 1 & 0  \\ 0 & a & 1 \\ 0 & 0 & a \end{pmatrix}, \quad
 H'  = \begin{pmatrix} a & 1 & 0  \\ 0 & a & 0 \\ 0 & 0 & a \end{pmatrix}
   \quad \text{and} \quad 
 H'' = \begin{pmatrix} a & 1 & 0  \\ 0 & a & 0 \\ 0 & 0 & b \end{pmatrix}
\end{equation*} 
with $a \neq b$, we have $\rho(H,E) = (3,2)$, $\rho(H',E) = (3,1)$ and $\rho(H'',E) = (2,2)$. Note that the three 
considered pencils are regular and, in particular, the eigenvalues of the first two are the same but their elementary divisors are not. 
In fact, their invariant polynomials are $\{(a+\lambda)^3,1,1\}$ for $H + \lambda E$, $\{(a+\lambda)^2,a+\lambda,1\}$ for $H' 
+ \lambda E$, and $\{(a+\lambda)^2,b+\lambda,1\}$ for $H'' + \lambda E$. 
\end{exmp}

\subsection{Induced hierarchy of matrix pencils}
\label{sec:hierarchy}

The tensor rank induces a straightforward hierarchy in any tensor space, namely, $\mathcal{S}_0 \subset \mathcal{S}_1 \subset 
\mathcal{S}_2 \subset \dots$, where $\mathcal{S}_r$ contains all tensors of rank up to $r$. Our definition of minimal ranks also 
induces a hierarchy which can be expressed by using the definition of the sets $\mathcal{B}_{r,s}$ given in \eqref{Bset}. However, 
such a hierarchy is more intricate, as now we have, for instance, $\mathcal{B}_{r,s} \subset 
\mathcal{B}_{r+1,s}$ and $\mathcal{B}_{r,s} \subset \mathcal{B}_{r,s+1}$ but $\mathcal{B}_{r+1,s} \not\subset 
\mathcal{B}_{r,s+1}$ and $\mathcal{B}_{r,s+1} \not\subset \mathcal{B}_{r+1,s}$.

\Cref{fig:hierarchy} contains a diagram depicting the hierarchy of $\mathcal{B}_{r,s}$ sets in the space of $m \times n$ real pencils, 
denoted by $\mathcal{P}_{m,n}(\mathbb{R})$. 
We assume $m \le n$ without loss of generality, since $\mathcal{P}_{m,n}(\mathbb{R})$ and $\mathcal{P}_{n,m}(\mathbb{R})$ have 
identical structures. 
For even $m$, the set $\mathcal{B}_{m,m} \subset \mathcal{P}_{m,n}(\mathbb{R})$ is always non-empty, but for odd $m$ the set 
$\mathcal{B}_{m,m} \subseteq \mathcal{P}_{m,n}(\mathbb{R})$ is non-empty if and only if $m < n$.
This is because an $m \times m$ pencil has full minimal ranks if and only if it is regular and its elementary divisors cannot be 
written as powers of linear forms, as we shall prove in the next section. In $\mathcal{P}_{m,m}(\mathbb{R} )$, this means that a 
pencil $A + \lambda B$ satisfies $\rho(A,B) = (m,m)$ if and only if it is $\GL_{m,m,2}(\mathbb{R})$-equivalent to some other pencil 
$Q + \lambda E$ where no eigenvalue of $Q$ is in $\mathbb{R}$. Since complex-valued eigenvalues of a real matrix necessarily 
arise in pairs, this can evidently only happen for even values of $m$.
For concreteness, three examples concerning $\mathcal{P}_{2,2}(\mathbb{R})$, $\mathcal{P}_{2,3}(\mathbb{R})$ and 
$\mathcal{P}_{3,3}(\mathbb{R})$ are shown in \Cref{fig:hierarchy}.

\begin{figure}[h!]
\centering
 \includegraphics[width=0.8\textwidth]{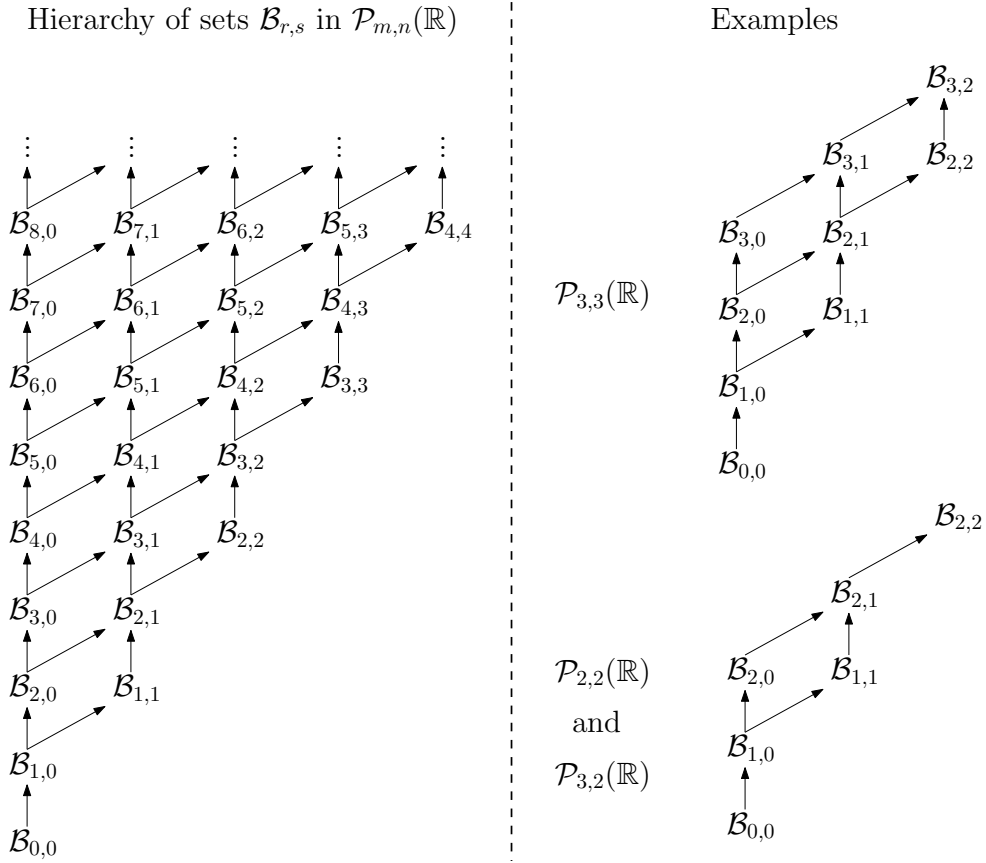}
\caption{Hierarchy of sets of real pencils in $\mathcal{P}_{m,n}(\mathbb{R})$ according to their minimal ranks: illustration of 
the general form (left) and concrete examples for three pencil spaces (right). The notation $\mathcal{B}_{r,s} \rightarrowTriangle 
\mathcal{B}_{r',s'}$ stands for $\mathcal{B}_{r,s} \subset \mathcal{B}_{r',s'}$. See \eqref{Bset} for a definition of the sets 
$\mathcal{B}_{r,s}$.}
 \label{fig:hierarchy}
\end{figure}

\subsection{Minimal ranks of Kronecker canonical forms}
\label{sec:Kronecker}

We now show how the minimal ranks of a pencil can be determined from its Kronecker canonical form. The notation $J_m(a) 
\triangleq a E_m + H_m$, where $H_m \triangleq \sum_{l=1}^{m-1} e_l \otimes e_{l+1}$, will be used for a 
Jordan block of size $m$ associated with the finite elementary divisor $(a + \lambda)^m$. In this definition, the vectors $e_l$ 
denote as usual the canonical basis vectors of their corresponding spaces. 
A canonical $v \times v$ block associated with an infinite elementary divisor will be expressed as $N_v(\lambda) \triangleq E_m 
+ \lambda H_v$.
Let us first consider regular pencils.

\begin{lem}
\label{lem:rho-regular} 
 Let $A + \lambda B$ be a regular $n \times n$ pencil and let 
$$ 
     A' + \lambda B' =  \left[ N_{v_1}(\lambda) \oplus \dots \oplus N_{v_j}(\lambda) \right] \oplus 
     \left[ (A_1 \oplus \dots \oplus A_k \oplus Q) + \lambda E \right]
$$ 
be its Kronecker canonical form, where the elementary divisors of $ Q + \lambda E_q$ cannot be factored into powers of linear forms
and $A_l = J_{m_l}(a_l)$ for some $a_l \in \mathbb{F}$, $l \in \{1,\ldots,k\}$. Let $k_s$ be the largest number of blocks $A_l$ whose 
elementary divisors share a common factor $a^\star + \lambda$ and $k_r$ be the second largest number of blocks $A_l$ whose elementary 
divisors share a common factor $a^{\star\star} + \lambda$ (with $a^{\star\star} \neq a^{\star}$). Then 
\begin{equation}
\label{rho-reg}
  \rho(A,B) = \left( n - k'_r, n - k'_s \right),
\end{equation}
where $k'_s$ and $k'_r$ are the first and second largest components of $(j,k_s,k_r)$, respectively. 
\end{lem}

\begin{proof}
By virtue of \Cref{prop:rho-orbit}, we have $\rho(A, B) = \rho(A',B')$. It thus suffices to show that $\rho(A',B') = (n - k'_r, n - 
k'_s)$. The steps are as follows.
\begin{enumerate}

\item First, we claim that $t Q + u E_q$ is nonsingular for any $(t,u) \neq 0$. 
This claim is trivially true if $t=0$ and $u \neq 0$. For $t \neq 0$ (and $u$ possibly null), the argument is as follows. Suppose for 
a contradiction that $t Q + u E_q$ is singular, with $t \neq 0$. Without loss of generality, we may take $t=1$. Then, there 
exists $P \in \GL_q(\mathbb{F})$ such that 
\begin{equation*}
 P \left( Q + u E_q \right) P^{-1} = P Q P^{-1} + u E_q = F \oplus J_p(0),
\end{equation*} 
where $0 < p \le q$. But then, $P Q P^{-1} = (F - u E_{q-p}) \oplus J_p(-u)$, implying $Q + \lambda E_q$ has an elementary 
divisor of the form $(-u + \lambda)^p$, which contradicts the hypothesis that the elementary divisors of $Q$ cannot be written as
powers of linear forms. As a consequence, $\rho(Q,E_q) = (q,q)$. In particular, if $Q = A'$ (i.e., $n = q$), then \eqref{rho-reg} 
yields $\rho(A',B') = (n,n)$ (because $j = k_s = k_r = 0$), as required.

\item  Now, note that $\rank\, (t J_{m}(a) + u E_m) < m$ for $(t,u) \neq 0$ if and only if $(t,u) = c(1,-a)$ for some $c \neq 0$, in 
which case $\rank\, (t J_{m}(a) + u E_m) = \rank J_{m}(0) = m-1$. Moreover, $\rank\, (t E_v + u H_v) < v$ for $(t,u) \neq 0$ if and 
only if $(t,u) = (0,c)$ for some $c \neq 0$, implying $\rank\, (t E_v + u H_v) = v-1$.
Hence, since by definition $k_s \ge k_r$, we have three cases:
\begin{enumerate}[(i)]

   \item If $k_s \ge j \ge k_r$, then we can apply the following transformation
      \begin{equation*}
      (A',B') \mapsto (B', A' -a^\star B') 
      \end{equation*}    
to obtain an $\GL_{n,n,2}(\mathbb{F})$-equivalent pencil attaining the minimal ranks of $(A',B')$. Indeed, $\rank \, (A' -a^\star B') 
= n - k_s$ is minimal among all linear combinations $t A' + u B'$ with $(t,u) \neq 0$. Hence, \eqref{s-def} must equal $n - k_s = n - 
k'_s$. Furthermore, $\rank B' = n - j$ is minimal among all linear combinations $t A' + u B'$ with $(t,u) \neq c (1, -a^\star)$. So, 
\eqref{r-def} must equal $r' = n - j = n - k'_r$.

   \item If $k_s \ge k_r \ge j$, then using a similar argument we deduce that the $\GL_{n,n,2}(\mathbb{F})$-equivalent pencil $(A' 
-a^{\star\star} B') + \lambda( A' -a^\star B' )$ attains the minimal ranks of $(A',B')$, showing that $\rho(A',B') = (n-k_r, n-k_s) 
= (n-k'_r, n-k'_s)$.

  \item Finally, if $j \ge k_s \ge k_r$, then following the same line of thought we have that $(A' -a^{\star} B') + \lambda B'$ 
attains the minimal ranks of $(A',B')$, that is, $\rho(A',B') = (n-k_r, n-j) = (n-k'_r, n-k'_s)$.

\end{enumerate}

\end{enumerate}
\end{proof}

For a singular $m \times n$ canonical pencil $A + \lambda B$ having no regular part, computing the minimal ranks is 
straightforward, because $a_{i,j}$ and $b_{i,j}$ cannot be both nonzero for any given pair of indices $(i,j)$. Indeed, the 
canonical block $L_k(\lambda)$ related to a minimal index $k$ associated with the columns is the $k \times (k+1)$ pencil of 
the form
\begin{equation*}
L_k(\lambda) \triangleq
 \begin{pmatrix}
    \lambda & 1 & & & \\
     & \lambda & 1 & & \\
     &  & \ddots & \ddots & \\
     & &  & \lambda & 1
 \end{pmatrix}.
\end{equation*} 
Any $\GL_2(\mathbb{F})$ transformation applied to $L_k(\lambda)$ yields some pencil $A' + \lambda B'$ such that $\rank 
A' = \rank B' = k$. In other words, $L_k(\lambda)$ has minimal ranks $(k,k)$. By the same argument, the canonical block $R_l(\lambda)$ 
related to a minimal index $l$ associated with the rows, which is an $(l+1) \times l$ pencil defined as $R_l(\lambda) \triangleq 
L^\T_l(\lambda)$, has minimal ranks $(l,l)$. The special case $k = 0$ (or $l = 0$) also adheres to that 
rule, as its minimal ranks are $(0,0)$.
Now, adjoining blocks having these forms yields a singular pencil 
$L_{k_1}(\lambda) \oplus \dots \oplus L_{k_p}(\lambda) \oplus R_{l_1}(\lambda) \oplus \dots \oplus R_{l_q}(\lambda)$ whose minimal 
ranks are clearly the sum of the minimal ranks of the blocks. Note that this is true even for the zero minimal indices $k_1 = \dots = 
k_{p'} = l_1 = \dots = l_{q'} = 0$, since they correspond to $p'$ null columns and $q'$ null rows, and so the minimal ranks must be 
bounded by $\min\{m - q', n-p'\}$. We have arrived at the following result. 

\begin{lem}
\label{lem:rho-singular}
Let $A + \lambda B$ be a singular $m \times n$ pencil having the form $A + \lambda B = L_{k_1}(\lambda) \oplus \dots \oplus 
L_{k_p}(\lambda) \oplus R_{l_1}(\lambda) \oplus \dots \oplus 
R_{l_q}(\lambda)$, where $k_1, \ldots, k_p$ are the minimal indices associated with its columns (henceforth called minimal column 
indices) and $l_1,\ldots,l_q$ are the minimal indices associated with its rows (minimal row indices). 
Then, \begin{equation*}
  \rho(A,B) = (\bar{s},\bar{s}), \quad \text{where } \bar{s} = k_1 + \dots + k_p + l_1 + \dots + l_q.
\end{equation*}
\end{lem}

For an arbitrary $m \times n$ pencil $A + \lambda B$, the block diagonal structure of its Kronecker canonical form allows a direct 
combination of the previous results, yielding the main theorem of this section.

\begin{theo}
\label{theo:rho-canonical}
Let $A + \lambda B$ be an arbitrary $m \times n$ pencil with Kronecker canonical form $S(\lambda) \oplus (A' + \lambda B')$, where 
$S(\lambda)$ is its singular part and $A' + \lambda B'$ is regular. Suppose $S(\lambda)$ has minimal column indices $k_1, \ldots, k_p$ 
and minimal row indices $l_1, \ldots, l_q$. Define $\bar{s} = k_1 + \dots + k_p + l_1 + \dots + l_q$. Then, its minimal ranks are given 
by
\begin{equation}
 \label{rho-canonical}
  \rho(A, B) = (r' + \bar{s}, s' + \bar{s}), 
\end{equation} 
where $(r', s') = \rho(A',B')$, whose components are given by \Cref{lem:rho-regular}.
\end{theo}

The above result implies that both minimal ranks of an $n \times n$ singular pencil must be strictly smaller than $n$. This is because 
the sum of the minimal indices of its singular part (which equals $\bar{s}$ in \eqref{rho-canonical}) can never attain the largest 
dimension of that part. Hence, if an $n \times n$ pencil has minimal ranks $(n,n)$ or $(n,n-1)$, then it is necessarily regular. 
In particular, $n \times n$ pencils with full minimal ranks can be characterized as follows.

\begin{corol}
\label{cor:full-min-rank}
An $n \times n$ pencil $A + \lambda B$ satisfies $\rho(A,B) = (n,n)$ if and only if it is regular and its elementary divisors cannot be 
written as powers of linear forms.
\end{corol}

\subsection{Classification of $\GL_{m,n,2}(\mathbb{R})$-orbits for $m,n \le 4$}
\label{sec:classif}

Using \Cref{theo:rho-canonical}, a complete classification of all Kronecker canonical forms of $m \times n$ pencils over 
$\mathbb{R}$ is provided in \Cref{tab:1,tab:2,tab:3,tab:4} for $m, n \in \{1,\ldots,4\}$. Each such form is associated with a 
family of $\GL_{m,n,2}(\mathbb{R})$-orbits. We denote canonical blocks whose elementary divisors are powers of second-order 
irreducible polynomials by
\begin{equation*}
  Q_{2k}(a,b) = E_k \kron \begin{pmatrix} a & b \\ -b & a \end{pmatrix} + J_2(0) \kron E_k \quad \text{with} \quad b \neq 0,
\end{equation*}
where $\kron$ denotes the Kronecker product. Observe that, because we consider orbits of $\GL_{m,n,2}(\mathbb{R})$ action, we 
can represent each family of orbits by a canonical form having no infinite elementary divisors (which can always be avoided by 
employing a $\GL_{2}(\mathbb{R})$ transformation).

It can be checked that each described family with $m, n \le 3$ corresponds to a single orbit,\footnote{Atkinson 
\cite{Atkinson1991} had already pointed out that, over an algebraically closed field $\mathbb{F}$, there are only finitely many 
$\GL_{3,n,2}(\mathbb{F})$-orbits for any $n$. Thus, over $\mathbb{R}$ this must be true of orbits whose elementary divisors are 
powers of linear forms.} except for $\mathscrc{R}'_{3,2}$, which encompasses an infinite number of non-equivalent orbits. All families 
having dimensions $m = 3$ and $n = 4$ (or $m = 4$ and $n = 3$) also contain only one orbit each. These properties can be 
verified by inspecting the equivalent pencils of \Cref{tab:equiv} shown ahead in \Cref{sec:equiv}: for $\min\{m,n\} \le 3$, the only 
family whose given canonical form depends on a parameter is that of $\mathscrc{R}'_{3,2}$.
For $m = n = 4$, infinitely many non-equivalent orbits are contained by each family in general. 
To avoid redundancies, families of orbits having zero minimal indices are omitted in the tables, since each such family corresponds 
to some other one of lower dimensions. For instance for $m=n=4$, if a singular pencil $A + \lambda B$ has minimal indices 
$k_1 = k_2 = l_1 = l_2 = 0$, then it can be expressed in the form $0 \oplus (A' + \lambda B')$, where both blocks in this 
decomposition have size $2 \times 2$. So, the canonical form of $A' + \lambda B'$ can be inspected to determine the properties 
of $A + \lambda B$. Similarly, not all combinations of canonical blocks are included for the singular part, because the shown 
properties remain the same if we transpose these blocks. To exemplify, it can be checked that $L_1 \oplus R_2$ and $R_1 \oplus L_2$ 
have the same dimensions, tensor rank, multilinear rank and minimal ranks, because the roles played by column and row minimal indices 
are essentially the same, up to a transposition.

A family is denoted with the letter $\mathscrc{R}$ or the letter $\mathscrc{S}$ if it encompasses regular or singular pencils, 
respectively. The subscript indices of each family indicate its minimal ranks, and primes are used to distinguish among otherwise 
identically labeled families. The tensor rank of each canonical form was determined using Corollary 2.4.1 of Ja'Ja' \cite{JaJa1979} 
and Theorem 4.6 of Sumi et al.~\cite{Sumi2009}, which requires taking into account the minimal indices of the pencil and also its 
elementary divisors. The values given in the column ``multilinear rank'' were determined by inspection; for a definition see 
\cite{deSilva2008}. Specifically, for an $m \times n$ pencil $A + \lambda B$ viewed as an $m \times n \times 2$ tensor $A 
\otimes e_1 + B \otimes e_2$, the multilinear rank is the triple $(r_1,r_2,r_3)$ satisfying
\begin{equation*}
 r_1 = \rank \begin{pmatrix} A & B \end{pmatrix},     \quad 
 r_2 = \rank \begin{pmatrix} A^\T & B^\T \end{pmatrix}, \quad 
 r_3 = \dim \sspan \{A, B\}.
\end{equation*} 
It should be noted that, in the above equation, $\sspan \{A, B\}$ denotes the subspace of $\mathbb{R}^{m \times n}$ spanned by 
the matrices $A$ and $B$, whose dimension is at most two. Finally, in order to determine the minimal ranks (column labeled ``$\rho$'') 
of each family, \Cref{theo:rho-canonical} was applied.

We point out that another classification of pencil orbits is given by Pervouchine 
\cite{Pervouchine2004}, but his study is concerned with closures of orbits and pencil bundles, not with tensor rank or minimal ranks. 
Our list is therefore a complement to the one he provides. Furthermore, the hierarchy of closures of pencil bundles he has 
presented bears no direct connection with the hierarchy of sets $\mathcal{B}_{r,s}$ we present in \Cref{sec:hierarchy}, which is 
easily determined by the numbers $r,s$ associated with each such set.

\begin{table}[t!]
\begin{center}
\def\arraystretch{1.6}
\begin{tabular}{c | c |c|c|c|c}
\hline 
\multirow{2}{*}{Family} & \multirow{2}{*}{Canonical form} & $m \times n$ & tensor & multilinear & \multirow{2}{*}{$\rho$} \\[-8pt]
& & &   rank   &    rank    &   \\
\hline
$\mathscrc{R}_{1,0}$ & $a + \lambda\, E_1$  & $1 \times 1$ & 1  & $(1,1,1)$ & $(1,0)$ \\[2pt] 
\hline 
$\mathscrc{S}_{1,1}$ & $L_1(\lambda)$ & $1 \times 2$ & 2 & $(1,2,2)$ & $(1,1)$ \\[2pt] 
\hline
\end{tabular}
\end{center}
 \caption{Families of canonical forms of $1 \times 1$ and $1 \times 2$ real pencils having no zero minimal indices, with $a_i 
\neq a_j$ for $i \neq j$.}
 \label{tab:1}
\end{table}

\begin{table}[t!]
\begin{center}
\def\arraystretch{1.6}
\begin{tabular}{c | c |c|c|c|c}
\hline 
\multirow{2}{*}{Family} & \multirow{2}{*}{Canonical form} & $m \times n$ & tensor & multilinear & \multirow{2}{*}{$\rho$} \\[-8pt] 
& & &   rank   &    rank    &   \\
\hline
$\mathscrc{R}_{1,1}$ & $a_1 \oplus a_2 + \lambda\, E_2$ & $2 \times 2$ & 2  & $(2,2,2)$ & $(1,1)$ \\ 
$\mathscrc{R}_{2,0}$ & $a \oplus a + \lambda\, E_2$ & $2 \times 2$  & 2  & $(2,2,1)$ & $(2,0)$ \\[2pt]
\hline 
$\mathscrc{R}_{2,1}$ & $J_2(a) + \lambda\, E_2$ & $2 \times 2$  & 3  & $(2,2,2)$ & $(2,1)$  \\[2pt] 
\hline 
$\mathscrc{R}_{2,2}$ & $Q_2(a,b) + \lambda\, E_2$ & $2 \times 2$  & 3  & $(2,2,2)$ & $(2,2)$  \\[2pt] 
\hline
$\mathscrc{S}_{2,2}$ & $L_2(\lambda)$ & $2 \times 3$  & 3  & $(2,3,2)$ & $(2,2)$  \\[2pt] 
$\mathscrc{S}_{2,1}$ & $L_1(\lambda) \oplus (a + \lambda\, E_1)$ & $2 \times 3$  & 3  & $(2,3,2)$ & $(2,1)$  \\[2pt] 
\hline
$\mathscrc{S}'_{2,2}$ & $L_1(\lambda) \oplus L_1(\lambda)$ & $2 \times 4$  & 4  & $(2,4,2)$ & $(2,2)$  \\[2pt] 
\hline
\end{tabular}
\end{center}
 \caption{Families of canonical forms of $2 \times 2$ and $2 \times 3$ real pencils having no zero minimal indices, with $a_i 
\neq a_j$ for $i \neq j$ and $b \neq 0$.}
 \label{tab:2}
\end{table}

\begin{table}[t!]
\begin{center}
\def\arraystretch{1.6}
\begin{tabular}{c | c |c|c|c|c}
\hline 
\multirow{2}{*}{Family} & \multirow{2}{*}{Canonical form} & $m \times n$ & tensor & multilinear & \multirow{2}{*}{$\rho$} \\[-8pt] 
& & &   rank   &    rank    &   \\
\hline 
$\mathscrc{S}''_{2,2}$ & $L_1(\lambda) \oplus R_1(\lambda)$ & $3 \times 3$ & 4  & $(3,3,2)$ & $(2,2)$ \\[2pt] 
\hline 
$\mathscrc{R}'_{2,2}$ & $a_1 \oplus a_2 \oplus a_3 + \lambda\, E_3$ & $3 \times 3$  & 3  & $(3,3,2)$ & $(2,2)$ \\ 
$\mathscrc{R}'_{2,1}$ & $a_1 \oplus a_2 \oplus a_2 + \lambda\, E_3$ & $3 \times 3$  & 3  & $(3,3,2)$ & $(2,1)$  \\ 
$\mathscrc{R}_{3,0}$ & $a \oplus a \oplus a + \lambda\, E_3$ & $3 \times 3$  & 3  & $(3,3,1)$ & $(3,0)$  \\[2pt] 
\hline 
$\mathscrc{R}_{3,2}$ & $J_3(a) + \lambda\, E_3$ & $3 \times 3$  & 4  & $(3,3,2)$ & $(3,2)$  \\ 
$\mathscrc{R}_{3,1}$ & $a \oplus J_2(a) + \lambda\, E_3$ & $3 \times 3$  & 4  & $(3,3,2)$ & $(3,1)$  \\ 
$\mathscrc{R}''_{2,2}$ & $a_1 \oplus J_2(a_2) + \lambda\, E_3$ & $3 \times 3$  & 4  & $(3,3,2)$ & $(2,2)$  \\[2pt]
\hline 
$\mathscrc{R}'_{3,2}$ & $a \oplus Q_2(c,d) + \lambda\, E_3$ & $3 \times 3$  &  4  & $(3,3,2)$ & $(3,2)$ \\[2pt] 
\hline
$\mathscrc{S}'''_{2,2}$ & $L_1(\lambda) \oplus (a_1 \oplus a_2 + \lambda\, E_2)$ & $3 \times 4$  &  4  & $(3,4,2)$ & $(2,2)$ \\ 
$\mathscrc{S}_{3,1}$ & $L_1(\lambda) \oplus (a \oplus a + \lambda\, E_2)$ & $3 \times 4$  &  4  & $(3,4,2)$ & $(3,1)$ \\ 
$\mathscrc{S}_{3,2}$ & $L_1(\lambda) \oplus (J_2(a) + \lambda\, E_2)$ & $3 \times 4$  &  5  & $(3,4,2)$ & $(3,2)$ \\ 
$\mathscrc{S}_{3,3}$ & $L_1(\lambda) \oplus (Q_2(a,b) + \lambda\, E_2)$ & $3 \times 4$  &  5  & $(3,4,2)$ & $(3,3)$ \\
$\mathscrc{S}'_{3,2}$ & $L_2(\lambda) \oplus (a + \lambda\, E_1)$ & $3 \times 4$  &  4  & $(3,4,2)$ & $(3,2)$ \\ 
$\mathscrc{S}'_{3,3}$ & $L_3(\lambda)$ & $3 \times 4$  &  4  & $(3,4,2)$ & $(3,3)$ \\[2pt] 
\hline
\end{tabular}
\end{center}
 \caption{Families of canonical forms of $3 \times 3$ and $3 \times 4$ real pencils having no zero minimal indices, with $a_i 
\neq a_j$ for $i \neq j$, $b \neq 0$ and $d \neq 0$.}
 \label{tab:3}
\end{table}

\begin{table}[t!]
\begin{center}
\def\arraystretch{1.6}
\begin{tabular}{c | c |c|c|c|c}
\hline 
\multirow{2}{*}{Family} & \multirow{2}{*}{Canonical form} & $m \times n$ & tensor & multilinear & \multirow{2}{*}{$\rho$} \\[-8pt]
& & &  rank   &    rank    &   \\
\hline 
$\mathscrc{S}''_{3,2}$ & $L_1(\lambda) \oplus R_1(\lambda) \oplus (a + \lambda\, E_1)$ & $4 \times 4$ & 5  & $(4,4,2)$ & $(3,2)$  
\\[2pt] 
$\mathscrc{S}''_{3,3}$ & $L_2(\lambda) \oplus R_1(\lambda) $ & $4 \times 4$ & 5  & $(4,4,2)$ & $(3,3)$  \\[2pt] 
\hline
$\mathscrc{R}_{4,0}$ & $a \oplus a \oplus a \oplus a + \lambda\, E_4$ & $4 \times 4$ & 4  & $(4,4,1)$ & $(4,0)$  \\ 
$\mathscrc{R}'''_{2,2}$ & $a_1 \oplus a_1 \oplus a_2 \oplus a_2 + \lambda\, E_4$ & $4 \times 4$ & 4 & $(4,4,2)$ & $(2,2)$  \\ 
$\mathscrc{R}'_{3,1}$ & $a_1 \oplus a_2 \oplus a_2 \oplus a_2 + \lambda\, E_4$ & $4 \times 4$ & 4  &  $(4,4,2)$ & $(3,1)$  \\ 
$\mathscrc{R}''_{3,2}$ & $a_1 \oplus a_2 \oplus a_3 \oplus a_3 + \lambda\, E_4$ & $4 \times 4$ & 4 &   $(4,4,2)$ & $(3,2)$  \\ 
$\mathscrc{R}_{3,3}$ & $a_1 \oplus a_2 \oplus a_3 \oplus a_4 + \lambda\, E_4$ & $4 \times 4$ & 4  & $(4,4,2)$ & $(3,3)$  \\[2pt] 
\hline
$\mathscrc{R}'''_{3,2}$ & $a_1 \oplus a_2 \oplus J_2(a_2) + \lambda\, E_4$ & $4 \times 4$ & 5  & $(4,4,2)$ & $(3,2)$  \\ 
$\mathscrc{R}''''_{3,2}$ & $a_1 \oplus a_1 \oplus J_2(a_2) + \lambda\, E_4$ & $4 \times 4$ & 5  & $(4,4,2)$ & $(3,2)$  \\
$\mathscrc{R}_{4,1}$ & $a \oplus a \oplus J_2(a) + \lambda\, E_4$ & $4 \times 4$ & 5  & $(4,4,2)$ & $(4,1)$  \\ 
$\mathscrc{R}'_{3,3}$ & $J_2(a_1) \oplus J_2(a_2) + \lambda\, E_4$ & $4 \times 4$ & 5  & $(4,4,2)$ & $(3,3)$  \\ 
$\mathscrc{R}''_{3,3}$ & $a_1 \oplus a_2 \oplus J_2(a_3) + \lambda\, E_4$ & $4 \times 4$ & 5  & $(4,4,2)$ & $(3,3)$  \\ 
$\mathscrc{R}'''_{3,3}$ & $a_1 \oplus J_3(a_2) + \lambda\, E_4$ & $4 \times 4$ & 5  & $(4,4,2)$ & $(3,3)$  \\ 
$\mathscrc{R}_{4,2}$ & $a \oplus J_3(a) + \lambda\, E_4$ & $4 \times 4$ & 5  & $(4,4,2)$ & $(4,2)$  \\ 
$\mathscrc{R}'_{4,2}$ & $J_2(a) \oplus J_2(a) + \lambda\, E_4$ & $4 \times 4$ & 6  & $(4,4,2)$ & $(4,2)$  \\ 
$\mathscrc{R}_{4,3}$ & $J_4(a) + \lambda\, E_4$ & $4 \times 4$ & 5  & $(4,4,2)$ & $(4,3)$  \\[2pt] 
\hline
$\mathscrc{R}''''_{3,3}$ & $a_1 \oplus a_2 \oplus Q_2(c,d) + \lambda\, E_4$ & $4 \times 4$ &  5  & $(4,4,2)$ & $(3,3)$  \\ 
$\mathscrc{R}''_{4,2}$ & $a \oplus a \oplus Q_2(c,d) + \lambda\, E_4$ & $4 \times 4$ &  5  & $(4,4,2)$ & $(4,2)$  \\
$\mathscrc{R}'_{4,3}$ & $J_2(a) \oplus Q_2(c,d) + \lambda\, E_4$ & $4 \times 4$ &  5  & $(4,4,2)$ & $(4,3)$  \\[2pt] 
\hline
$\mathscrc{R}_{4,4}$ & $Q_2(a,b) \oplus Q_2(c,d) + \lambda\, E_4$ & $4 \times 4$ &  5  & $(4,4,2)$ & $(4,4)$ \\ 
$\mathscrc{R}'_{4,4}$ & $Q_4(a,b) + \lambda\, E_4$ & $4 \times 4$ &  5  & $(4,4,2)$ & $(4,4)$ \\ 
$\mathscrc{R}''_{4,4}$ & $Q_2(a,b) \oplus Q_2(a,b) + \lambda\, E_4$ & $4 \times 4$ &  6  & $(4,4,2)$ & $(4,4)$ \\[2pt] 
\hline
\end{tabular}
\end{center}
 \caption{Families of canonical forms of $4 \times 4$ real pencils having no zero minimal indices, with $a_i 
\neq a_j$ for $i \neq j$, $b \neq 0$ and $d \neq 0$.}
 \label{tab:4}
\end{table}

\subsection{Minimal ranks of matrix polynomials}
\label{sec:matrix-poly}

One can generalize \Cref{def:min-ranks} to matrix polynomials of finite degree as follows. The minimal ranks of a 
degree-$(d-1)$ polynomial $P(\lambda) = \sum_{k=1}^{d} \lambda^{k-1} A_k$ should correspond to the minimal values $r_1 \ge \dots \ge 
r_d$ such that we can find 
a transformation $T \in \GL_d(\mathbb{F})$ for which $(E,E,T) \cdot P(\lambda) = \sum_{k=1}^{d} \lambda^{k-1} A'_k$ where $\rank A'_k 
= r_k$ for $k=1,\ldots,d$. One can thus introduce the \emph{rank-minimizing subspaces} $\mathcal{T}_1, \ldots, \mathcal{T}_l$ of 
$\mathbb{F}^d$ with respect to $P(\lambda)$, where $1 \le l \le d$, which are inductively defined in the following manner. First, 
let $\mathcal{T}_1$ be the subspace of $\mathbb{F}^d$ spanned by all solutions of 
\begin{equation*}
   \min_{ (t_{1}, \ldots, t_{d}) \neq 0 } \rank \left(  \sum_{k=1}^{d} t_{k} A_k \right) = \bar{r}_1.
\end{equation*} 
If $\dim \mathcal{T}_1 = d$, then $\mathcal{T}_1$ is the only rank-minimizing subspace, that is, $l=1$. Otherwise, we define next 
$\mathcal{T}_2$ as the subspace of $\mathbb{F}^d$ spanned by all solutions of 
\begin{equation*}
   \min_{ (t_{1}, \ldots, t_{d}) \notin \mathcal{T}_1 } \rank \left(  \sum_{k=1}^{d} t_{k} A_k \right) = \bar{r}_2.
\end{equation*} 
If $\dim \mathcal{T}_1 + \dim \mathcal{T}_2 = d$, then we have $\mathcal{T}_1 \oplus \mathcal{T}_2 = \mathbb{F}^2$ and $l=2$. 
Otherwise, one continues in this fashion until $\dim \mathcal{T}_1 + \dots + \dim \mathcal{T}_l = d$ for some $l$, which must happen 
after finitely many steps. So, each $\mathcal{T}_p$ is defined as the subspace spanned by the solutions of 
\begin{equation}
 \label{ms-Tp}
   \min_{ (t_{1}, \ldots, t_{d}) \notin \bar{\mathcal{T}} } 
     \rank \left(  \sum_{k=1}^{d} t_{k} A_k \right) = \bar{r}_p,
     \quad \text{where} \quad
     \bar{\mathcal{T}} = 
     \begin{cases}
       \{ 0 \}, & p = 1, \\
       \mathcal{T}_1 \oplus \dots \oplus \mathcal{T}_{p-1}, & 1 < p \le l.
     \end{cases}
\end{equation} 
It is clear that $\mathcal{T}_1 \oplus \dots \oplus \mathcal{T}_l = \mathbb{F}^d$. The minimal ranks are then associated with this 
decomposition, as defined below.

\begin{defn}
\label{def:min-ranks-poly}
Let $ P(\lambda) = \sum_{k=1}^{d} \lambda^{k-1} A_k$ be an $m \times n$ matrix polynomial over $\mathbb{F}$ of degree (at most) $d-1$, 
and let $\mathcal{T}_1, \ldots, \mathcal{T}_l$ be the rank-minimizing subspaces of $\mathbb{F}^d$ associated with $P(\lambda)$, with 
$\dim \mathcal{T}_p = d_p$.
The minimal ranks of $P(\lambda)$, denoted by $\rho(P)$, are defined as the components of the $d$-tuple 
\begin{equation*}
 \rho(P) = (r_1,\ldots,r_d) \triangleq ( 
            \underbrace{\bar{r}_l,\ldots,\bar{r}_l}_{d_l \text{ times}},
            \underbrace{\bar{r}_{l-1},\ldots,\bar{r}_{l-1}}_{d_{l-1} \text{ times}},
            \ldots,
            \underbrace{\bar{r}_2,\ldots,\bar{r}_2}_{d_2 \text{ times}} 
            \underbrace{\bar{r}_1,\ldots,\bar{r}_1}_{d_1 \text{ times}} 
            ),
\end{equation*}
where the numbers $\bar{r}_p$ are given by \eqref{ms-Tp} and satisfy $\bar{r}_l > \bar{r}_{l-1} > \dots > \bar{r}_1$.
We say that the components of $\rho(P)$ which equal $\bar{r}_p$ are associated with $\mathcal{T}_p$.
\end{defn}

For $d=2$, \Cref{def:min-ranks-poly} is equivalent to \Cref{def:min-ranks}. In particular, $\rho(A,B) = (r,s)$ satisfies $r = s$ if 
and only if $\mathbb{F}^2$ has a single associated rank-minimizing subspace $\mathcal{T}_1 = \mathbb{F}^2$ with respect to $A + 
\lambda B$. An analogue of \Cref{lem:GL2} also clearly holds for matrix polynomials.

Let us introduce 
\begin{equation*}
\label{Bset-poly}
 \mathcal{B}_{r_1,\ldots,r_d} \triangleq 
                   \left\{ P(\lambda) 
                   \; \Bigg| \;
                   \exists \, \sum_{k=1}^{d} \lambda^{k-1} A'_k \in \mathcal{O}(P) 
                   \mbox{ such that }
                   \rank A_k' \le r_k \mbox{ for } k=1,\ldots,d \right\},
\end{equation*}
where $\mathcal{O}(P)$ stands for the $\GL_{m,n,d}$-orbit of $P(\lambda)$.
$\mathcal{B}_{r_1,\ldots,r_d}$ is clearly invariant with respect to a permutation of $r_1,\ldots,r_d$, and thus we shall assume 
$r_1 \ge \dots \ge r_d$. 
Assuming $\rho(P) = (r_1,\ldots,r_d)$, the construction of the rank-minimizing subspaces shows there is a transformation $T \in 
\GL_d(\mathbb{F})$ yielding $(E,E,T) \cdot P(\lambda) = \sum_{k=1}^{d} \lambda^{k-1} A'_k$ such that $\rank A'_k = r_k$,
and thus by definition $P(\lambda) \in \mathcal{B}_{r_1,\ldots,r_d}$. 
Now, since the rows of any other $T' \in \GL_d(\mathbb{F})$ must span $\mathbb{F}^d$, it is easy to see 
that $P(\lambda)$ cannot belong to any  $\mathcal{B}_{r_1,\ldots,r_{q-1},r'_q,r_{q+1},\ldots,r_d}$  such that $r_q > r'_q \ge 
r_{q+1}$. Indeed, this would contradict the construction of the subspaces $\mathcal{T}_p$ as being spanned by \emph{all} solutions of 
the rank minimization problem \eqref{ms-Tp}.
In fact, if $s_1 \ge \dots \ge s_d$ and $P(\lambda) \in \mathcal{B}_{s_1,\ldots,s_d}$, then we must have $s_q \ge r_q$ for all $q \in 
\{1,\ldots,d\}$, because the assumption $r_q > s_q \ge s_{q+1} \ge \dots \ge s_d$ is inconsistent with the above construction 
of the rank-minimizing subspaces $\mathcal{T}_1,\ldots,\mathcal{T}_l$ of $\mathbb{F}^d$ with respect to $P(\lambda)$. This is the 
central argument of the following generalization of \Cref{prop:r-s}.

\begin{prop}
Let $ P(\lambda) = \sum_{k=1}^{d} \lambda^{k-1} A_k$ be an $m \times n$ matrix polynomial over $\mathbb{F}$.
If $\rho(P) = (r_1,\ldots,r_d)$ and $s_1 \ge \dots \ge s_d$, then $P(\lambda) \in \mathcal{B}_{s_1,\ldots,s_d}$ if and only if 
$s_k \ge r_k$ for all $k \in \{1,\ldots,d\}$.
\end{prop}

In view of these extensions, it makes sense to consider the hierarchy of $m \times n$ matrix polynomials of any finite degree 
in terms of their minimal ranks, leading to a diagram such as that of \Cref{fig:hierarchy}.

\subsection{Decomposing third-order tensors into matrix-vector tensor products of minimal ranks}
\label{sec:MVD}

The connection between $m \times n$ matrix pencils and $m \times n \times 2$ tensors has been exploited time and 
again to derive many results, such as the characterization of $\GL_{m,n,2}(\mathbb{R})$-orbits in terms of their tensor ranks carried 
out by Ja'Ja' and Sumi et al.~\cite{JaJa1979, Sumi2009}. This correspondence is also well-suited to study block-term 
decompositions of $m \times n \times 2$ tensors composed by matrix-vector products, as we can directly associate the ranks of these 
blocks with the ranks of the matrices constituting the pencil.

More generally, a natural connection exists between a block-term decomposition of an $m \times n \times d$ tensor and $m \times n$ 
polynomial of degree $d-1$. Given any third order tensor $X \in \mathbb{F}^{m} \otimes \mathbb{F}^{n} \otimes \mathbb{F}^{d}$, it can 
always be written in the form
\begin{equation*}
  X = \sum_{k=1}^d \left( \sum_{j=1}^{r_k} u^{(k)}_j \otimes v^{(k)}_j \right) \otimes w_k,
\end{equation*} 
where the vectors $w_1, \ldots, w_d$ are linearly independent. Choosing coordinates for these spaces, this expression can be 
identified with
\begin{equation}
 \label{BTD}
  X = \sum_{k=1}^d A_k \otimes w_k \quad \in \mathbb{F}^{m \times n \times d},
\end{equation} 
where $A_k$ is a matrix of rank at most $r_k$. By suitably defining an isomorphism $\mathbb{F}^d \simeq 
\mathbb{F}_{d-1}[\lambda]$, where $\mathbb{F}_{d-1}[\lambda]$ denotes the space of degree-$(d-1)$ polynomials, the tensor in 
\eqref{BTD} can be associated with the matrix polynomial
\begin{equation*}
  P(\lambda) = \sum_{k=1}^d \lambda^{k-1} A_k.
\end{equation*}

From this link, it becomes evident that the minimal ranks of $P(\lambda)$ characterize the intrinsic complexity of the tensor $X$ in 
\eqref{BTD} in terms of the ranks of the matrices appearing in the sum. It thus quantifies the dimensions of the smooth manifolds 
whose join set contains $X$; see the recent work by Breiding,  and Vannieuwenhoven \cite{Breiding2018} for a discussion on this 
interpretation of \eqref{BTD}. In the case of a matrix pencil $P(\lambda) = A + \lambda B$ ($d=2$), our results in 
\Cref{sec:Kronecker} allow one to compute $\rho(A,B)$ from the Kronecker canonical form of the pencil. Similar characterizations for 
higher values of $d$ would be surely valuable for the study of the BTD and its properties.

\section{A pencil may have no best approximation with strictly lower minimal ranks}
\label{sec:ill-posed}

This section investigates the existence of optimal approximations of a pencil $A + \lambda B \in \mathcal{P}_{m,n}(\mathbb{F})$ on 
$\mathcal{B}_{r,s}$, for a given pair $(r,s)$ satisfying $r \ge s$, for $\mathbb{F} = \mathbb{R}$ or $\mathbb{F} = 
\mathbb{C}$. 
Defining such approximations requires a topology for $\mathcal{P}_{m,n}(\mathbb{F})$. We shall pick the norm topology, which is 
the same regardless of the chosen norm, since $\mathcal{P}_{m,n}(\mathbb{F})$ has a finite dimension. It can be introduced by 
considering the inner product
\begin{equation}
 \label{inner}
   \langle A + \lambda B, C + \lambda D \rangle \triangleq \langle A , C  \rangle + \langle B, D \rangle, 
\end{equation}
where the inner product appearing in the right-hand side is the standard Euclidian one given by
\begin{equation*}
   \langle A , C \rangle \triangleq \trace A^* C, 
\end{equation*}
with $A^*$ denoting the adjoint of $A$.  This leads to the Euclidian norm
\begin{equation*}
  \| A + \lambda B \| \triangleq \sqrt{ \langle A + \lambda B , A + \lambda B \rangle  }
                              = \sqrt{ \|A\|^2 + \|B\|^2},
\end{equation*}
which induces the topology. In the above expression, $\|A\| \triangleq \sqrt{\langle A, A \rangle}$ is the Frobenius norm of $A$.
With this definition, the approximation problem can be formulated as
\begin{equation}
 \label{best-app}
  \inf_{ A' + \lambda B' \in \mathcal{B}_{r,s} }  \| A + \lambda B - (A' + \lambda B' ) \|.
\end{equation}
We are thus interested in determining whether this infimum is attained for a given pencil $A + \lambda B$ and some choice of ranks $(r,s)$. 
Evidently, this question is only of interest at all when $A + \lambda B \not\in \mathcal{B}_{r,s}$. To check whether this holds, we 
shall rely on the concept of minimal ranks and its associated results developed in \Cref{sec:pencils}.

In view of the isomorphism between $\mathcal{P}_{m,n}(\mathbb{F})$ and $\mathbb{F}^{m \times n \times 2}$ discussed in \Cref{sec:MVD}, 
the results of the present section apply also to the approximation of tensors from $\mathbb{F}^{m \times n \times 2}$ by a sum 
of two matrix-vector tensor products, in the tensor norm topology.
To define this topology, one can consider the Frobenius norm 
\begin{equation*}
  \| A \otimes e_1 + B \otimes e_2 \| \triangleq \sqrt{\langle A \otimes e_1 + B \otimes e_2 , A \otimes e_1 + B \otimes e_2 \rangle  }
%   = \sqrt{\| A \|^2 + \| B \|^2}
\end{equation*}
where the scalar product is defined for rank-one tensors as
\begin{equation*}
  \langle u \otimes v \otimes w, \, x \otimes y \otimes z  \rangle \triangleq \langle u, x \rangle \cdot \langle v, y \rangle \cdot \langle w, z \rangle
\end{equation*}
and extends bi-linearly to tensors of arbitrary rank. Therefore, 
$$ \| A \otimes e_1 + B \otimes e_2 \| = \sqrt{\|A\|^2 + \|B\|^2} = \| A + \lambda B \|.$$

\subsection{A template for ill-posed pencil approximation problems}
\label{sec:ill-posed-examples}

Examples of tensors having no best rank-$r$ approximation in the norm topology have been known for quite a while; see 
\cite{Comon2014, deSilva2008} and references therein. De Lathauwer \cite{deLathauwer2008c} has employed the same kind of 
construction to provide an example of a tensor having no best approximate block-term decomposition constituted by two blocks of 
multilinear rank $(2,2,2)$. In the next proposition, we resort to a similar expedient to derive a template of ill-posed instances 
of problem \eqref{best-app} of a certain kind.

\begin{prop}
\label{prop:exist-PL}
 Let $A, B$ be $m \times s$ matrices and $C, D$ be $n \times s$ matrices and suppose 
\begin{equation}
\label{cond-ex}
\min\left\{ \, \rank \begin{pmatrix} A & B \end{pmatrix}, \, 
      \rank \begin{pmatrix} C & D \end{pmatrix} \,\right\} > r =  \frac{3s}{2} .  
\end{equation} 
 Then, the pencil 
 \begin{equation}
   P(\lambda) = \left( A C^\T + B D^\T \right) + \lambda \, B C^\T \quad \in \mathcal{P}_{m,n}(\mathbb{F})
 \end{equation}
has no best approximation in $\mathcal{B}_{s,s}$.
\end{prop}

\begin{proof}
First, let us show that for any transformation $T \in \GL_2(\mathbb{F})$, at least one of the matrices $t_{11} (A C^\T + B D^\T) + 
t_{12} B C^\T$ and $t_{21} (A C^\T + B D^\T) + t_{22} B C^\T$ has rank strictly larger than $s$. 
Note that we can write  $t_{i1} (A C^\T + B D^\T) + t_{i2} B C^\T = F_i \begin{pmatrix} C & D \end{pmatrix}^\T$, where
$$ F_i = \begin{pmatrix} A & B \end{pmatrix} \begin{pmatrix} t_{i1} E_s & 0 \\ t_{i2} E_s & t_{i1} E_s \end{pmatrix}. $$
We have $\rank F_i = \rank \begin{pmatrix} A & B \end{pmatrix} > r$ for $t_{i1} \neq 0$, and so Sylvester's inequality implies that 
the product $F_i \begin{pmatrix} C & D \end{pmatrix}^\T$ has rank strictly greater than $2(r - s) = s$.
Because $t_{11}$ and $t_{21}$ obviously cannot be both zero, the statement is true. Hence, we 
conclude that $P(\lambda) \notin \mathcal{B}_{s,s}$.
Next, let
 \begin{align}
 \label{Pn}
P_n(\lambda) \triangleq & \ n 
  \left[ \left( B + \frac{1}{n} A \right)\left( C + \frac{1}{n} D \right)^\T  \right] 
  \left( 1 + \frac{1}{n} \lambda \right) - n \left( B C^\T \right)  \\
    = & \ P(\lambda) + \frac{1}{n}\left[  \left( A C^\T \right) \lambda +
                 \left( B D^\T \right) \lambda +
                 \left( A D^\T \right) \right] 
                 + \frac{1}{n^2} \left( A D^\T \right) \lambda.  
 \label{Pn2}
 \end{align}
By construction, $P_n(\lambda) \in \mathcal{B}_{s,s}$, while \eqref{Pn2} reveals that $P_n(\lambda) \rightarrow P(\lambda) $ as 
$n \rightarrow \infty$. Hence, since $P(\lambda) \notin \mathcal{B}_{s,s}$, it holds that 
\begin{equation*}
 \inf_{A' + \lambda B' \in \mathcal{B}_{s,s}} \, 
    \| P(\lambda) - (A' + \lambda B') \| = 0
\end{equation*} 
is not attained.
\end{proof}

\begin{rem}
 The condition \eqref{cond-ex} is tight in the sense that one can find matrices $A$, $B$, $C$ and $D$ satisfying $\rank 
\begin{pmatrix} A & B\end{pmatrix} \le r$ and $\rank \begin{pmatrix} C & D \end{pmatrix} \le r$ such that 
 $P(\lambda) = \left( A C^\T + B D^\T \right) + \lambda \, B C^\T \in \mathcal{B}_{s,s}$. Take, for instance, $m, n \ge 6$, $s=4$ and
\begin{align*}
   A = \begin{pmatrix} a_1 & a_2 & a_3 & a_4 \end{pmatrix}, \quad & \quad
   B = \begin{pmatrix} a_1 & a_2 & b_3 & b_4 \end{pmatrix}, \\
   C = \begin{pmatrix} c_1 & c_2 & c_3 & c_4 \end{pmatrix}, \quad & \quad
   D = \begin{pmatrix} d_1 & d_2 & c_3 & c_4 \end{pmatrix},
\end{align*}
where $a_1, a_2, a_3, a_4, b_3, b_4$ are linearly independent, and the same applies to $c_1, c_2, c_3, c_4, d_1, d_2$. We have then
$\rank \begin{pmatrix} A & B \end{pmatrix} = \rank \begin{pmatrix} C & D \end{pmatrix} = 6 = \frac{3s}{2} = r$. Choosing $t_{i1} = 1$ 
and $t_{i2} = -1$ yields $F_i \begin{pmatrix} C & D \end{pmatrix}^\T  = (a_3 - b_3) c_3^\T + (a_4 - b_4) c_4^\T + B D^\T $, which 
clearly cannot have rank larger than $s = 4$.
\end{rem}

\subsection{Ill-posedness over a positive-volume set of real pencils}
\label{sec:positive-volume}

In this section, we will prove that no (regular) $2k \times 2k$ pencil having only complex-valued eigenvalues admits a best 
approximation in $\mathcal{B}_{2k-1,2k-1}$ in the norm topology, for any positive integer $k$. The set containing all such pencils is 
defined as 
\begin{align}
 \label{C-set}
  \mathcal{C} \triangleq & \ \{ A + \lambda B \in \mathcal{P}_{2k,2k}(\mathbb{R}) \ | \ \mathcal{O}(A,B) \text{ contains } Q + \lambda 
E, 
   \text{ where } Q  \text{ has no real eigenvalues }\} \\
   = & \ \{ A + \lambda B \in \mathcal{P}_{2k,2k}(\mathbb{R}) \ | \ \rho(A,B) = (2k,2k)\},  \nonumber
\end{align} 
where the equality is due to \Cref{cor:full-min-rank}. 
For instance, in the case $k=2$ this set is constituted by all orbits of the families $\mathscrc{R}_{4,4}$, $\mathscrc{R}'_{4,4}$ and 
$\mathscrc{R}''_{4,4}$ of \Cref{tab:4}. We start by showing $\mathcal{C}$ is open, and therefore has positive volume (since 
it is always nonempty).

\begin{lem}
\label{lem:C-open}
The set $\mathcal{C} \subset \mathcal{P}_{2k,2k}(\mathbb{R})$ defined by \eqref{C-set} is open in the norm topology.
\end{lem}

\begin{proof}
Take an arbitrary pencil $A + \lambda B \in \mathcal{C}$. By definition, it can be written as $(P,U,T) \cdot (Q + \lambda E)$,
where $(P,U,T) \in \GL_{2k,2k,2}(\mathbb{R})$ and 
$Q = Q_{2 k_1}(a_1,b_1) \oplus Q_{2 k_2}(a_2,b_2) \oplus \dots \oplus Q_{2 k_l}(a_l,b_l)$, with $k_1 + \dots + k_l = k$ and 
$b_i \neq 0$ for $i=1,\ldots,l$.
Consider any other pencil $A' + \lambda B'$ lying in an open ball of radius $\sqrt{\epsilon}$ centered on $A + \lambda B$. We 
have
\begin{equation*}
  \left\| (P,U,T) \cdot (Q + \lambda E) - (A' + \lambda B') \right\|^2
= \left\| (P,U,T) \cdot \left(Q - A'' + \lambda(E - B'') \right) \right\|^2 < \epsilon,
\end{equation*} 
where $A'' + \lambda B'' = (P^{-1}, U^{-1}, T^{-1}) \cdot (A' + \lambda B') \not\equiv Q + \lambda E$. But, since $(P,U,T) \in 
\GL_{2k,2k,2}(\mathbb{R})$, then we have
$
  \| (P,U,T) \cdot (C + \lambda D) \| \ge \sigma \|C + \lambda D\|
$ 
for any $C + \lambda D \in \mathcal{P}_{2k,2k}(\mathbb{R})$, where $\sigma > 0$ is the smallest singular value of the linear 
operator $(P,U,T) : \mathcal{P}_{2k,2k}(\mathbb{R}) \rightarrow \mathcal{P}_{2k,2k}(\mathbb{R})$. So, 
\begin{equation*}
    \left\| Q - A'' \right\|^2 + \left\| E - B'' \right\|^2 
\le \sigma^{-2} \left\| (P,U,T) \cdot \left(Q - A'' + \lambda(E - B'') \right) \right\|^2
 <  \sigma^{-2} \, \epsilon.
\end{equation*}
Hence, a sufficiently small $\epsilon$ can be chosen to guarantee that $\|Q - A''\| \le \epsilon_1$ and $\| E - B'' \| \le \epsilon_2$ 
for any $\epsilon_1, \epsilon_2 > 0$. By continuity of the eigenvalues of the pencil $Q + \lambda E$, it follows that there 
exists $\epsilon > 0$ such that every such $A'' + \lambda B''$ can be written as $A'' + \lambda B'' = \left( X, Y, Z \right) \cdot 
(Q' + \lambda E)$ for some $(X,Y,Z) \in \GL_{2k,2k,2}(\mathbb{R})$ and $Q'$ having $k$ pairs of complex conjugate eigenvalues 
(with possibly some identical pairs). Therefore, $A' + \lambda B' = \left( PX, UY, TZ \right) \cdot (Q' + \lambda E) \in 
\mathcal{C}$. Because this applies to every $A' + \lambda B'$ in the chosen open ball of radius $\epsilon$, then $A + \lambda B$ is an 
interior point of $\mathcal{C}$.
\end{proof}

\begin{corol}
\label{cor:Y-bestapp}
 The set $\mathcal{B}_{2k,2k-1} \subset \mathcal{P}_{2k,2k}(\mathbb{R})$ is closed in the norm topology. Consequently, 
 $$
  \inf_{A' + \lambda B' \in \mathcal{B}_{2k,2k-1}} \| A + \lambda B - (A' + \lambda B') \|
 $$
 is always attained by some $A' + \lambda B' \in \mathcal{B}_{2k,2k-1}$.
\end{corol}

\begin{proof}
Follows from \Cref{lem:C-open} by using the fact that $\mathcal{B}_{2k,2k-1} = \mathcal{P}_{2k,2k}(\mathbb{R}) \setminus 
\mathcal{C}$.
\end{proof}

Another way to define $\mathcal{C}$ is by stating that it contains all $2k \times 2k$ pencils $A + \lambda B$ such that $\det(u A + t 
B) = 0$ if and only if $t=u=0$. This corresponds to the set of absolutely nonsingular $2k \times 2k \times 2$ tensors defined by Sumi 
et al.~in their paper \cite{Sumi2013}. It was shown in \cite[Theorem 2.5]{Sumi2013} that, for any positive integer $k$ and $d > 1$, 
the set of absolutely nonsingular $2k \times 2k \times d$ tensors is open. \Cref{lem:C-open} therefore provides an alternative proof 
of that fact for the case $d = 2$.

\begin{theo}
\label{theo:closure}
 The set $\mathcal{B}_{2k-1,2k-1} \subset \mathcal{P}_{2k,2k}(\mathbb{R})$ is not closed in the norm topology. Furthermore, its 
closure is given by $ \overline{\mathcal{B}_{2k-1,2k-1}} =  \mathcal{B}_{2k,2k-1} \subset \mathcal{P}_{2k,2k}(\mathbb{R})$.
\end{theo}

\begin{proof}
To prove this claim, it is sufficient\footnote{Indeed, if $\mathcal{X} \subset \mathcal{Y}$, $\mathcal{Y} = 
\overline{\mathcal{Y}}$ and $\mathcal{Y} \subseteq \overline{\mathcal{X}}$, then $\overline{\mathcal{X}} \subseteq 
\overline{\mathcal{Y}}$ and thus $\mathcal{Y} \subseteq \overline{\mathcal{X}} \subseteq \overline{\mathcal{Y}} = \mathcal{Y}$, 
implying $\overline{\mathcal{X}} = \mathcal{Y}$.} to show that for every pencil $A + \lambda B \in \mathcal{B}_{2k,2k-1} \subset 
\mathcal{P}_{2k,2k}(\mathbb{R})$ having minimal ranks $\rho(A,B) = (2k,2k-1)$, we can find a sequence of pencils in 
$\mathcal{B}_{2k-1,2k-1}$ which converges to $A + \lambda B$. 
First, recall that every pencil $A + \lambda B$ with $\rho(A,B) = (2k,2k-1)$ must be regular.
From \Cref{lem:rho-regular}, $\rho(A,B) = (2k,2k-1)$ holds precisely when the canonical form of $A + \lambda B$ comprises 
either a single Jordan block $J_{m_1}(b)$ for some $b \in \mathbb{R}$, or a single block $N_{m_1}(\lambda)$ having an infinite 
elementary divisor, with even dimension $m_1 > 0$. Let us first focus on the case with no infinite elementary divisors, where the 
canonical form is either of the form $J_{2k}(a) + \lambda E$ (with $m_1 = 2k$) or of the form $J_{m_1}(a) \oplus Q + \lambda 
E$, where the elementary divisors of $Q + \lambda E$ cannot be factored into powers of linear forms. The latter possibility 
exists of course only when $k > 1$. Let us consider each case separately:
\begin{enumerate}[(i)]

 \item In the former case, define $Z_p(\lambda) \triangleq W_p + \lambda E$, where 
 \begin{equation*}
W_p \triangleq J_{2k-1}(a) \oplus \left(a + \frac{1}{p}\right) + e_{2k-1} \otimes e_{2k} = 
\begin{pmatrix}
a & 1 &  &  &  &  \\
 & a & 1 &  & &  \\
  &   & \ddots &  \ddots &   & \\
 &  &  & a & 1 &  \\
 &  & &  & a & 1 \\  
 &  & &  &  & a + 1/p \\
\end{pmatrix},
\end{equation*}
where $p$ is a positive integer. It is not hard to see that $Z_p(\lambda)$ is $\GL_{2k,2k,2}(\mathbb{R})$-equivalent to 
$J_{2k-1}(a) \oplus (a + 1/p) + \lambda E$, since their elementary divisors coincide.
Hence, by \Cref{lem:rho-regular}, $Z_p(\lambda)$ has minimal ranks $(2k-1,2k-1)$, and thus belongs to $\mathcal{B}_{2k-1,2k-1}$. 
On the other hand, $A + \lambda B = \left( P, U, T  \right) \cdot (J_{2k}(a) + \lambda E)$ for some $\left( P, U, T  \right) 
\in \GL_{2k,2k,2}(\mathbb{R})$, and so we have
\begin{equation*}
   \lim_{p \rightarrow \infty}
    \left\| A + \lambda B - \left( P, U, T  \right) \cdot Z_p(\lambda) \right\|
 = \lim_{p \rightarrow \infty}
     \left\| \left( P, U, T  \right) \cdot \left( J_{2k}(a) - W_p \right) \right\| = 0.
\end{equation*} 
It follows that the sequence of pencils $\{ \left( P, U, T  \right) \cdot Z_p(\lambda)\}_{p \in \mathbb{N}}$ converges to $A + \lambda 
B$ as $p \rightarrow \infty$.

\item The same argument can be employed in the second case, now with $Z'_p(\lambda) \triangleq W'_p + \lambda E$, where
$$ W'_p \triangleq J_{m_1 - 1}(a) \oplus (a + 1/p) \oplus Q + e_{m_1 - 1} \otimes e_{m_1}.$$
From this definition, $Z_p(\lambda)$ is $\GL_{2k,2k,2}(\mathbb{R})$-equivalent to $J_{m_1-1}(a) \oplus (a + 1/p) \oplus 
Q + \lambda E$, which is also in $\mathcal{B}_{2k-1,2k-1}$. Writing again $A + \lambda B = \left( P, U, T  \right) \cdot (J_{m_1}(a) 
\oplus Q + \lambda E)$, it follows that
\begin{equation*}
   \lim_{p \rightarrow \infty}
     \left\| A + \lambda B - \left( P, U, T  \right) \cdot Z'_p(\lambda) \right\|
 = \lim_{p \rightarrow \infty}
     \left\| \left( P, U, T  \right) \cdot \left( J_{m_1}(a) \oplus Q - W'_p \right) \right\| = 0.
\end{equation*}
\end{enumerate}
Finally, note that the above argument extends easily to the case where $A + \lambda B$ has one infinite divisor, because there 
still exists $(P,U,T) \in \GL_{m,n,2}(\mathbb{R})$ such that either $A + \lambda B = \left( P, U, T  \right) \cdot (J_{2k}(a) + \lambda 
E)$ or $A + \lambda B = \left( P, U, T  \right) \cdot (J_{m_1}(a) \oplus Q + \lambda E)$, since infinite elementary divisors can always 
be avoided with a $\GL_{2}(\mathbb{R})$ transformation.
\end{proof}

Motivated by the fact that pencils with minimal ranks $(2k,2k-1)$ can be arbitrarily well approximated by pencils having minimal ranks 
$(2k-1,2k-1)$, one may define the ``border minimal ranks'' of a pencil in the same fashion as the border rank is defined for tensors 
(see, e.g., \cite{deSilva2008}). Specifically, \Cref{theo:closure} shows that every real $2k \times 2k$ pencil with minimal ranks 
$(2k,2k-1)$ has ``border minimal ranks'' $(2k-1,2k-1)$.
As a consequence, no pencil with minimal ranks $(2k,2k-1)$ has a best approximation $\mathcal{B}_{2k-1,2k-1}$ in the norm topology, as 
stated next.

\begin{corol}
\label{corol:approx-43} 
 If $A + \lambda B \in \mathcal{P}_{2k,2k}(\mathbb{R})$ satisfies $\rho(A,B) = (2k,2k-1)$, then  
 \begin{equation*}
  \inf_{A' + \lambda B' \in \mathcal{B}_{2k-1,2k-1}} \| A + \lambda B - (A' + \lambda B') \| = 0
 \end{equation*}
is not attained by any $A' + \lambda B' \in \mathcal{B}_{2k-1,2k-1}$.
\end{corol}

The next result establishes that the best approximation of any pencil $A + \lambda B \in \mathcal{C}$ on $\mathcal{B}_{2k,2k-1}$ 
must have minimal ranks $(2k,2k-1)$, otherwise it is not optimal. It is thus in the same spirit of Lemma 8.2 of De Silva and Lim 
\cite{deSilva2008}, which states that for positive integers $r,s$ such that $r \ge s$, the best approximation of a rank-$r$ tensor 
having rank up to $s$ always has rank $s$.  

\begin{lem}
\label{lem:approx-43}
 Let $A + \lambda B \in \mathcal{C}$. Then, in the norm topology we have 
 \begin{equation*}
  \arg\min_{A' + \lambda B' \in \mathcal{B}_{2k,2k-1}} \| A + \lambda B - (A' + \lambda B') \|
  \quad \subset \quad \mathcal{B}_{2k,2k-1} \setminus  \left( \mathcal{B}_{2k,2k-2} \cup \mathcal{B}_{2k-1,2k-1} \right).
 \end{equation*} 
 In other words, every best approximation $A' + \lambda B'$ of $A + \lambda B$ on $\mathcal{B}_{2k,2k-1}$ is such that 
$\rho(A',B') = (2k,2k-1)$. 
\end{lem}

\begin{proof}
Take any $A' + \lambda B' \in \mathcal{B}_{2k,2k-2} \cup \mathcal{B}_{2k-1,2k-1}$. By definition, it can be written as
$(t_{11} U + t_{12} V) + \lambda (t_{21} U + t_{22} V)$, with $T = (t_{ij}) \in \GL_2(\mathbb{R})$ and either $(\rank U, \rank V) \le 
(2k-1,2k-1)$ or $(\rank U, \rank V) \le (2k,2k-2)$, where the inequality is meant entry-wise. Let us define now $Z_1(\lambda) 
\triangleq 
t_{11}(A- A') + \lambda t_{21}(B- B')$ and $Z_2(\lambda) \triangleq t_{12}(A- A') + \lambda t_{22}(B- B')$. Since $A' + \lambda B' 
\not\equiv A + \lambda B$, then $Z_i(\lambda) \not\equiv 0$ for at least one $i \in \{1,2\}$. Thus, there exists a rank-one matrix $W$ 
such that $\langle Z_i(\lambda), W + \lambda W \rangle \neq 0$ for a certain $i \in \{1,2\}$. Without loss of generality, we can 
assume that $\|t_{1i} W + \lambda t_{2i} W \|^2 = (t_{1i}^2 + t_{2i}^2) \|W\|^2 = 1$. For any $c \in \mathbb{R}$, we have
\begin{multline*}
 \| A + \lambda B - (A' + \lambda B' + c \, (t_{1i} W + \lambda t_{2i} W) ) \|^2 \\
 = \| A + \lambda B - (A' + \lambda B') \|^2 - 2c \langle Z_i(\lambda) , W + \lambda W \rangle + c^2.
\end{multline*} 
Now, if we choose $c = \langle Z_i(\lambda) , W + \lambda W \rangle  \neq 0$, then clearly 
\begin{equation*}
 \| A + \lambda B - (A' + \lambda B' + c \, (t_{1i} W + \lambda t_{2i} W) ) \|^2
  = \| A + \lambda B - (A' + \lambda B') \|^2 - c^2  < \| A + \lambda B - (A' + \lambda B') \|^2.
\end{equation*}
This means that $A' + \lambda B' + c \, (t_{1i} W + \lambda t_{2i} W)$ is closer to $A + \lambda B$ than $A' + \lambda B'$. Using now 
the expressions given for $A'$ and $B'$, we find that
\begin{equation*}
A'' + \lambda B'' \triangleq
A' + \lambda B' + c \, (t_{1i} W + \lambda t_{2i} W) =
    (t_{11} U + t_{12} V + c t_{1i} W) + \lambda (t_{21} U + t_{22} V + c t_{2i} W).
\end{equation*}
If $i=1$, we have $ (E,E,T^{-1}) \cdot (A'' + \lambda B'') = (U + c W) + \lambda V $, otherwise 
$ (E,E,T^{-1}) \cdot (A'' + \lambda B'') = U + \lambda (V + c  W) $. Either way, since $\rank W = 1$, then $A'' + \lambda B'' 
\in \mathcal{B}_{2k,2k-1}$. This shows that for any $A' + \lambda B' \in \mathcal{B}_{2k-1,2k-1} \cup \mathcal{B}_{2k,2k-2}$, we can 
always find some other pencil $A'' + \lambda B''$ in $\mathcal{B}_{2k,2k-1}$ such that $\| A + \lambda B - (A'' + \lambda B'') 
\| < \| A + \lambda B - (A' + \lambda B') \|$. Because a best approximation of $A + \lambda B \in \mathcal{B}_{2k,2k-1}$ must exist 
due to \Cref{cor:Y-bestapp}, we conclude that it can only belong to  $\mathcal{B}_{2k,2k-1} \setminus \left( \mathcal{B}_{2k-1,2k-1} 
\cup \mathcal{B}_{2k,2k-2}\right)$, i.e., it necessarily has minimal ranks $(2k,2k-1)$.
\end{proof}

We are now in a position to prove the main result of this section.

\begin{theo}
In the norm topology, if $A + \lambda B \in \mathcal{C}$ then 
 \begin{equation*}
  \inf_{A' + \lambda B' \in \mathcal{B}_{2k-1,2k-1}} \| A + \lambda B - (A' + \lambda B') \| 
 \end{equation*}
is not attained by any $A' + \lambda B' \in \mathcal{B}_{2k-1,2k-1}$. In other words, the problem above stated has no solution 
when $\rho(A,B)=(2k,2k)$. 
\end{theo}

\begin{proof}
 \Cref{cor:Y-bestapp} and \Cref{lem:approx-43} imply that for every pencil $A + \lambda B \in \mathcal{C}$ there exists another 
pencil $A' + \lambda B'$ such that $\rho(A',B') = (2k,2k-1)$ and
\begin{equation}
\label{B43-B33}
 \left\| A + \lambda B - (A' + \lambda B') \right\| < \left\| A + \lambda B - (C + \lambda D) \right\|, 
 \quad \forall \, C + \lambda D \in \mathcal{B}_{2k-1,2k-1} \subset \mathcal{B}_{2k,2k-1}. 
\end{equation}
But, it follows from \Cref{theo:closure} that any such $A' + \lambda B'$ can be arbitrarily well approximated by pencils from 
$\mathcal{B}_{2k-1,2k-1}$, that is, 
\begin{equation}
   \inf_{C + \lambda D \in \mathcal{B}_{2k-1,2k-1}} \| A' + \lambda B' - (C + \lambda D) \| = 0, 
\end{equation}
whilst no $C + \lambda D \in \mathcal{B}_{2k-1,2k-1}$ can attain that infimum because $A' + \lambda B' \notin \mathcal{B}_{2k-1,2k-1}$.
Combining the above facts, we conclude that
\begin{equation}
   \inf_{C + \lambda D \in \mathcal{B}_{2k-1,2k-1}} \| A + \lambda B - (C + \lambda D) \| = \| A + \lambda B - (A' + \lambda B') \|
\end{equation} 
cannot be attained by any $C + \lambda D \in \mathcal{B}_{2k-1,2k-1}$.
\end{proof}

 \section{Conclusion}
\label{sec:concl}

This work defines and studies a fundamental property of a matrix pencil, which we have called its minimal ranks. 
The structure of a space of pencils can be better understood on the basis of this notion and its properties.
In particular, endowing $\mathcal{P}_{m,n}( \mathbb{F} )$ with a norm, we have studied the problem of approximation of a 
pencil by another one having strictly lower minimal ranks in the induced norm topology. An optimal approximation may not exist, and 
our results show that this is true for every pencil of the set $\mathcal{C} \subset \mathcal{P}_{2k,2k}( \mathbb{R} )$ if an 
approximation is sought over $\mathcal{B}_{2k-1,2k-1}$ for any positive integer $k$. $\mathcal{C}$ is open, which shows that, 
contrarily to the complex-valued case, this phenomenon can happen for pencils forming a positive volume set. Translated to a tensor 
viewpoint, our result states that certain $2k \times 2k \times 2$ real tensors forming a positive-volume set have no best approximate 
block-term decomposition with two rank-$(2k-1)$ blocks. 

As we have shown, the definition and essential properties of the minimal ranks can be readily extended to matrix polynomials, which 
are associated with more general $m \times n \times d$ tensors. We believe this should provide a useful element for the study of 
third-order block-term decompositions composed by matrix-vector tensor products. In particular, results enabling the computation 
of the minimal ranks of a general matrix polynomial would certainly be helpful in this regard.

\appendix

\section{Equivalent pencils of Kronecker canonical forms}
\label{sec:equiv}

In \Cref{tab:equiv}, we give $\GL_{m,n,2}(\mathbb{R})$-equivalent pencils for each canonical form listed on Tables 
\ref{tab:1}--\ref{tab:3}. It can be seen that the only pencil with a free parameter is that in the family $\mathscrc{R}'_{3,2}$, 
which thus comprises infinitely many orbits. For all other listed families, a single orbit exists.

\begin{table}
\centering
   \begin{tabular}{c|c}
   \hline
   \multicolumn{2}{c}{Table \ref{tab:1}} \\
   \hline
     Family & Equivalent pencil \\
   \hline
$\mathscrc{R}_{1,0}$ & $E_1$ \\
\hline
$\mathscrc{S}_{1,1}$ & $L_1(\lambda)$ \\       
 \hline 
 \multicolumn{2}{c}{\ }  \\[10pt]
   \hline
   \multicolumn{2}{c}{Table \ref{tab:2}} \\
   \hline
     Family & Equivalent pencil \\   
   \hline   
$\mathscrc{R}_{1,1}$& $(E_1 + \lambda \,  0) \oplus (0 + \lambda \,  E_1)$ \\ 
$\mathscrc{R}_{2,0}$& $E_2$ \\
\hline
$\mathscrc{R}_{2,1}$& $J_2(0) + \lambda \,  E_2$ \\
\hline
$\mathscrc{R}_{2,2}$& $Q_2(0,1) + \lambda \,  E_2$ \\
\hline
$\mathscrc{S}_{2,2}$& $L_2(\lambda)$ \\
$\mathscrc{S}_{2,1}$& $L_1(\lambda) \oplus E_1$ \\
\hline
$\mathscrc{S}'_{2,2}$& $L_1(\lambda) \oplus L_1(\lambda)$  \\
\hline
   \end{tabular}
\hspace{0.5cm}
   \begin{tabular}{c|c}
   \hline
   \multicolumn{2}{c}{Table \ref{tab:3}} \\
   \hline
     Family & Equivalent pencil \\   
   \hline   
$\mathscrc{S}''_{2,2}$& $L_1(\lambda) \oplus R_1(\lambda)$ \\ 
\hline
$\mathscrc{R}'_{2,2}$& $(E_2 + \lambda \,  (E_1 \oplus 0)) \oplus (0 + \lambda \,  E_1)$ \\ 
$\mathscrc{R}'_{2,1}$& $(E_1 + \lambda \, 0) \oplus (0 + \lambda \,  E_2)$ \\ 
$\mathscrc{R}_{3,0}$& $E_3$ \\ 
\hline
$\mathscrc{R}_{3,2}$& $J_3(0) + \lambda \,  E_3$ \\ 
$\mathscrc{R}_{3,1}$& $(0 \oplus J_2(0)) + \lambda \,  E_3$ \\ 
$\mathscrc{R}''_{2,2}$& $ (0 + \lambda \,  E_1) \oplus (J_2(1) + \lambda \,  J_2(0))$ \\ 
\hline
$\mathscrc{R}'_{3,2}$& $(a' \oplus Q_2(0,1)) + \lambda \,  E_3$ \\ 
\hline
$\mathscrc{S}'''_{2,2}$& $L_1(\lambda) \oplus (E_1 + \lambda \, 0) \oplus (0 + \lambda \,  E_1)$ \\ 
$\mathscrc{S}_{3,1}$& $L_1(\lambda) \oplus (E_2 + \lambda \, 0)$ \\ 
$\mathscrc{S}_{3,2}$& $L_1(\lambda) \oplus (J_2(0) + \lambda \,  E_2)$ \\ 
$\mathscrc{S}_{3,3}$& $L_1(\lambda) \oplus (Q_2(0,1) + \lambda \,  E_2)$ \\
$\mathscrc{S}'_{3,2}$& $L_2(\lambda) \oplus (E_1 + \lambda \, 0)$ \\
$\mathscrc{S}'_{3,3}$& $L_3(\lambda)$ \\
\hline
   \end{tabular}
 \caption{Equivalent pencils to each canonical form of Tables \ref{tab:1}--\ref{tab:3}. Zero blocks are explicitly written for 
clarity.}
 \label{tab:equiv}   
\end{table}

%-----------------------------------------%

%-----------------------------------------%

\section*{Acknowledgements}

The authors are grateful for several suggestions given by Y.~Qi, which allowed in particular to significantly improve 
\Cref{sec:positive-volume}.
J.~H.~de M.~Goulart would like to thank also S.~Icart, J.~Lebrun and R.~Cabral Farias for helpful discussions on the subject.

\bibliographystyle{abbrv}
\bibliography{goulart-comon-pencils}

\end{document}